\documentclass[a4paper,10pt]{article}

\usepackage{listings}
\usepackage{textcomp}
\usepackage{amsfonts}
\usepackage{amsthm}
\usepackage{amsmath}
\usepackage{amssymb}
\usepackage{mathtools}
\usepackage{url}
\usepackage{enumerate}
\usepackage{array}
\usepackage{datetime}

\usepackage[top=1in, bottom=1in, left=1in, right=1in]{geometry}

\usepackage{graphicx}
\usepackage{epstopdf}
\usepackage{sidecap}
\usepackage{faktor}
\usepackage{units}

\usepackage{stmaryrd}

\usepackage{bbm}

\usepackage{color}

\usepackage{authblk}

\usepackage{sectsty}
\sectionfont{\sc\Large}

\newtheorem{thm}{Theorem}[section]
\newtheorem{lem}{Lemma}[section]
\newtheorem{coroll}{Corollary}[section]

\theoremstyle{definition}
\newtheorem{defn}{Definition}[section]
\newtheorem{example}{Example}[section]

\DeclareSymbolFont{pxfontssymbolsC}{U}{pxsyc}{m}{n}
\DeclareMathSymbol{\coloneqq}{\mathrel}{pxfontssymbolsC}{66}

\begin{document}

\title{On some new properties of fractional derivatives with Mittag-Leffler kernel}

\date{}


\author[1,2]{Dumitru Baleanu\thanks{Email: \texttt{dumitru@cankaya.edu.tr}}}
\author[3]{Arran Fernandez\thanks{Email: \texttt{af454@cam.ac.uk}}}

\affil[1]{{\small Department of Mathematics, Cankaya University, 06530 Balgat, Ankara, Turkey}}
\affil[2]{{\small Institute of Space Sciences, Magurele-Bucharest, Romania}}
\affil[3]{{\small Department of Applied Mathematics and Theoretical Physics, University of Cambridge, Wilberforce Road, CB3 0WA, United Kingdom}}

\maketitle

\vspace{-1cm}

\begin{abstract}\noindent We establish a new formula for the fractional derivative with Mittag-Leffler kernel, in the form of a series of Riemann--Liouville fractional integrals, which brings out more clearly the non-locality of fractional derivatives and is easier to handle for certain computational purposes. We also prove existence and uniqueness results for certain families of linear and nonlinear fractional ODEs defined using this fractional derivative. We consider the possibility of a semigroup property for these derivatives, and establish extensions of the product rule and chain rule, with an application to fractional mechanics.
\end{abstract}








\section{Introduction}
\label{sec-intro}
\textit{Fractional calculus} is the study of generalised orders of differentiation and integration (together referred to as \textit{differintegration}): beyond integer orders to real numbers, complex numbers, and beyond \cite{podlubny,samko,magin}.

There are many different definitions of fractional differintegrals, many of which conflict with each other in some regions of their domains of definition. One of the most commonly used is the \textit{Riemann--Liouville} formula, in which the $\alpha$th integral of a function $f$ is given by \cite{podlubny,samko}
\begin{equation}
\label{RL-int-defn}
D_{c+}^{-\alpha}f(t)\coloneqq\tfrac{1}{\Gamma(\alpha)}\int_c^t(t-\tau)^{\alpha-1}f(\tau)\,\mathrm{d}\tau,
\end{equation}
where $\alpha>0$ (so that this is an integral rather than a derivative) and $c$ is a constant of integration, generally taken to be either $0$ or $-\infty$. Fractional Riemann--Liouville \textit{derivatives} are then defined to be standard derivatives of fractional RL integrals:
\begin{equation}
\label{RL-deriv-defn}
D_{c+}^{\alpha}f(t)\coloneqq\frac{\mathrm{d}^n}{\mathrm{d}t^n}\Big({}_cD_t^{\alpha-n}f(t)\Big), \;n\coloneqq\lfloor\alpha\rfloor+1,
\end{equation}
for any $\alpha>0$.

Another frequently used definition of fractional differintegrals is the \textit{Caputo} one \cite{podlubny,samko}. Here, fractional integrals are again defined using the formula (\ref{RL-int-defn}), while fractional derivatives are defined to be fractional RL integrals of standard derivatives:
\begin{equation}
\label{Caputo-deriv-defn}
D_{c+}^{\alpha}f(t)\coloneqq D_{c+}^{\alpha-n}\Big(\tfrac{\mathrm{d}^n}{\mathrm{d}t^n}f(t)\Big), \;n\coloneqq\lfloor\alpha\rfloor+1,
\end{equation}
for any $\alpha>0$.

Fractional derivatives and integrals of this Riemann--Liouville type have a vast number of applications across many fields of science and engineering \cite{valerio}. For example, they can be used to model controllability \cite{liu}, viscoelastic flows \cite{mainardi, bazhlekova}, chaotic systems \cite{petras}, Stokes problems \cite{yu}, thermoelasticity \cite{povstenko}, several vibration and diffusion processes \cite{atanackovic, sousa, al-refai}, bioengineering problems \cite{magin}, and many other complex phenomena.

A natural question now is to investigate which properties of fractional derivatives make them so suitable to model successfully certain complex systems from several branches of science and engineering. The answer to this question relies  on the property exhibited by many of the aforementioned systems of \textit{non-local} dynamics: that is, the processes and dynamics possess a certain degree of memory, and the fractional operators defined above are non-local, while the ordinary derivative is clearly a local operator. Thus it can be useful to consider fractional-order models in any problem involving global optimisation, where some degree of memory is expected; many such problems are found for example in control theory.

We recall that the fractional integrator, as well as the related continuous frequency distributed differential model, represents an important tool for the simulation of fractional systems together with the solution of initial condition problems \cite{lorenzo,trigeassou1,trigeassou2,trigeassou3}. We stress the fact that the infinite-dimensional state vector of fractional integrators gives a direct generalisation to fractional calculus of the theoretical results corresponding to the integer order systems.

Some applications when the above mentioned results appear are in control theory \cite{razminia} and fractional variational principles \cite{atanackovic}, e.g. on fractional Euler-Lagrange and Hamilton equations and the fractional generalisation of total time derivatives \cite{baleanu2}.

Recently, Caputo and Fabrizio \cite{caputo} have proposed a new definition of fractional derivatives: \[\prescript{CF}{}D_{c+}^{\alpha}f(x)\coloneqq\tfrac{M(\alpha)}{1-\alpha}\frac{\mathrm{d}}{\mathrm{d}t}\int_c^x\exp\big[\tfrac{-\alpha}{1-\alpha}(x-y)\big]f(y)\,\mathrm{d}y,\] valid for $0<\alpha<1$, with $M(\alpha)$ being a normalisation function satisfying $M(0)=M(1)=1$. The basic challenge they were addressing was whether it is possible to construct another type of fractional operator which has nonsingular kernel and which can better describe in some cases the dynamics of non-local phenomena. The Caputo--Fabrizio definition has already found applications in areas such as diffusion modelling \cite{hristov} and mass-spring-damper systems \cite{al-salti}.

Here we shall mostly be considering a more recently developed definition for fractional differintegrals, due to \cite{atangana}. This new type of calculus addresses the same underlying challenge as that of Caputo and Fabrizio, but it uses a kernel which is non-local as well as non-singular, namely the Mittag-Leffler function:
\begin{align*}
\prescript{ABR}{}D^{\alpha}_{a+}f(t)&=\tfrac{B(\alpha)}{1-\alpha}\frac{\mathrm{d}}{\mathrm{d}t}\int_a^tf(x)E_\alpha\Big(\tfrac{-\alpha}{1-\alpha}(t-x)^{\alpha}\Big)\,\mathrm{d}x, \\
\prescript{ABC}{}D^{\alpha}_{a+}f(t)&=\tfrac{B(\alpha)}{1-\alpha}\int_a^tf'(x)E_\alpha\Big(\tfrac{-\alpha}{1-\alpha}(t-x)^{\alpha}\Big)\,\mathrm{d}x,
\end{align*}
again valid for $0<\alpha<1$ with $B(\alpha)$ being a normalisation function. (We shall examine these definitions and the required assumptions in more detail below.) In this way we are able to describe a different type of dynamics of non-local complex systems. In fact the classical fractional calculus and the one corresponding to the Mittag-Leffler nonsingular kernel complement each other in the attempt to better describe the hidden aspects of non-local dynamical systems. Fractional calculus based on the non-singular Mittag-Leffler kernel is more easily used from the numerical viewpoint, and this has been studied for example in \cite{djida}.

We note that the Mittag-Leffler function is already known to be highly useful in fractional calculus \cite{mathai}. It has also been an important part of the definitions of certain other fractional differintegrals, such as that proposed by Pskhu \cite{pskhu}.

Applications of the new AB formula have been explored in fields as diverse as chaos theory \cite{alkahtani}, heat transfer \cite{atangana}, and variational problems \cite{abdeljawad2}. Furthermore, it is natural to address the same questions about the fractional integrator and applications of these new operators in the theory of control and related fractional variational Euler-Lagrange and Hamilton equations (see \cite{abdeljawad2,djida,gomez-aguilar}). Besides, we expect to obtain some new terms in all generalised formulae from the classical fractional calculus, and this aspect will be important for the related applications.

Some basic properties of the new AB differintegrals have already been proven in several recent papers: for example, the original paper \cite{atangana} established the formulae for Laplace transforms of AB differintegrals and some Lipschitz-type inequalities; the paper \cite{abdeljawad2} considered integration by parts identities and Euler-Lagrange equations; and the paper \cite{djida} established, using Laplace transforms, analogues of the Newton--Leibniz formula for the integral of a derivative. However, much of the ground-level theory of this new model of fractional calculus has not yet been fully developed, and this paper aims to add to this basic theory by establishing new fundamental results in the field.

The structure of our paper is as follows. In section \ref{sec-series} we prove a new series formula for AB derivatives, which provides a more direct connection to non-locality properties of fractional calculus, and which will likely also be useful in numerical applications, since we can truncate the series at some finite point without having to deal with the transcendental Mittag-Leffler function explicitly. In section \ref{sec-ode} we get the ball rolling with the theory of fractional differential equations, by solving some general classes of linear and nonlinear ODEs in the AB model, with reference to some possible applications for these ODEs. In section \ref{sec-semigroup} we consider the semigroup property, an important point of interest in any model of fractional calculus, and the conditions for its validity in the AB model. In sections \ref{sec-product} and \ref{sec-chain}, we use the series formula of section \ref{sec-series} to prove extensions of the product rule and chain rule to the AB model, and briefly explore an application of the chain rule to fractional dynamical systems.

\section{A new formula for the fractional derivative with Mittag-Leffler kernel}
\label{sec-series}
The following definitions for fractional derivatives are established in \cite{atangana}.

\begin{defn}
\label{ABRL-defn}
The ABR fractional derivative (R denotes Riemann--Liouville type) is defined by
\begin{equation}
\label{ABRL}
\prescript{ABR}{}D^{\alpha}_{a+}f(t)=\tfrac{B(\alpha)}{1-\alpha}\frac{\mathrm{d}}{\mathrm{d}t}\int_a^tf(x)E_\alpha\Big(\tfrac{-\alpha}{1-\alpha}(t-x)^{\alpha}\Big)\,\mathrm{d}x
\end{equation}
for $0<\alpha<1$, $a<t<b$, and $f\in L^1(a,b)$.
\end{defn}

\begin{defn}
\label{ABC-defn}
The ABC fractional derivative (C denotes Caputo type) is defined by
\begin{equation}
\label{ABC}
\prescript{ABC}{}D^{\alpha}_{a+}f(t)=\tfrac{B(\alpha)}{1-\alpha}\int_a^tf'(x)E_\alpha\Big(\tfrac{-\alpha}{1-\alpha}(t-x)^{\alpha}\Big)\,\mathrm{d}x
\end{equation}
for $0<\alpha<1$, $a<t<b$, and $f$ a differentiable function on $[a,b]$ such that $f'\in L^1(a,b)$.
\end{defn}

In the above definitions, the function $E_{\alpha}$ is the Mittag-Leffler function, defined by:
\begin{equation}
\label{ML}
E_{\alpha}(x)=\sum_{n=0}^{\infty}\frac{x^n}{\Gamma(\alpha n+1)}.
\end{equation}
In general, the normalisation function $B(\alpha)$ can be any function satisfying $B(0)=B(1)=1$, but for the present work we shall assume that all values of $B(\alpha)$ are real and strictly positive. The reason for introducing this multiplier function, often taken simply to be identically $1$, is because sometimes it is desired to integrate over the order of differentiation $\alpha$, e.g. in some modelling problems. We may wish some values of $\alpha$ to contribute more than others, and hence we allow ourselves leeway to weight different values of $\alpha$ if we so desire.

\begin{lem}
\label{AB-fn-spaces}
For given $\alpha,a,t\in\mathbb{R}$ with $a<t$ and $0<\alpha<1$, and a given normalisation function $B$, the ABR derivative $\prescript{ABR}{}D^{\alpha}_{a+}f(t)$ is well-defined for any function $f$ such that the RL integral $D_{a+}^{-\alpha}f(t)$ is well-defined, while the ABC derivative $\prescript{ABC}{}D^{\alpha}_{a+}f(t)$ is well-defined for any differentiable function $f$ such that $f'$ is an $L^1$ function.
\end{lem}

\begin{proof}
We first consider the ABR derivative, and differentiate explicitly:
\begin{align}
\nonumber \prescript{ABR}{}D^{\alpha}_{a+}f(t)&=\tfrac{B(\alpha)}{1-\alpha}\frac{\mathrm{d}}{\mathrm{d}t}\int_a^tf(x)E_\alpha\Big(\tfrac{-\alpha}{1-\alpha}(t-x)^{\alpha}\Big)\,\mathrm{d}x \\
\nonumber &=\frac{B(\alpha)}{1-\alpha}\Bigg[f(t)E_\alpha\Big(\tfrac{-\alpha}{1-\alpha}(0)^{\alpha}\Big)+\int_a^tf(x)\frac{\mathrm{d}}{\mathrm{d}t}\Bigg(E_\alpha\Big(\tfrac{-\alpha}{1-\alpha}(t-x)^{\alpha}\Big)\Bigg)\,\mathrm{d}x\Bigg] \\
\label{ABRL-condn} &=\frac{B(\alpha)}{1-\alpha}\Bigg[f(t)+\int_a^tf(x)\sum_{n=0}^{\infty}\frac{n\alpha\big(\tfrac{-\alpha}{1-\alpha}\big)^n(t-x)^{n\alpha-1}}{\Gamma(n\alpha+1)}\,\mathrm{d}x\Bigg]
\end{align}
The function $E_\alpha\big(\tfrac{-\alpha}{1-\alpha}(t-x)^{\alpha}\big)$ and its $t$-derivative, considered as functions of $x$, are holomorphic at every point in the interval $[a,t)$. And the interval of integration is finite, so the only way the integral could possibly diverge would be due to behaviour near $x=t$. Thus the conditions for the ABR derivative to be well-defined are exactly that the integral in (\ref{ABRL-condn}) should behave well as $x\rightarrow t$ from below.

As $x\rightarrow t$, we have $(t-x)^{\alpha}\rightarrow0$ and therefore \[E_\alpha\Big(\tfrac{-\alpha}{1-\alpha}(t-x)^{\alpha}\Big)\sim1+\tfrac{-\alpha}{(1-\alpha)\Gamma(\alpha+1)}(t-x)^{\alpha},\] giving \[\frac{\mathrm{d}}{\mathrm{d}t}\Bigg(E_\alpha\Big(\tfrac{-\alpha}{1-\alpha}(t-x)^{\alpha}\Big)\Bigg)\sim\tfrac{-\alpha^2}{(1-\alpha)\Gamma(\alpha+1)}(t-x)^{\alpha-1}.\] So the integral in (\ref{ABRL-condn}) converges if and only if \[\int_a^tf(x)(t-x)^{\alpha-1}\,\mathrm{d}x\] converges, i.e. if and only if the RL integral $D_{a+}^{-\alpha}f(t)$ is well-defined.

Now we consider the ABC derivative:
\begin{equation}
\label{ABC-condn}
\prescript{ABC}{}D^{\alpha}_{a+}f(t)=\tfrac{B(\alpha)}{1-\alpha}\int_a^tf'(x)E_\alpha\Big(\tfrac{-\alpha}{1-\alpha}(t-x)^{\alpha}\Big)\,\mathrm{d}x
\end{equation}
Once again, the function $E_\alpha\big(\tfrac{-\alpha}{1-\alpha}(t-x)^{\alpha}\big)$ is holomorphic as a function of $x$ at every point in the interval $[a,t)$, and the interval of integration is finite. So the conditions for the ABC derivative to be well-defined are exactly that this integral should behave well as $x\rightarrow t$ from below. We also know that as $x\rightarrow t$, \[E_\alpha\Big(\tfrac{-\alpha}{1-\alpha}(t-x)^{\alpha}\Big)\sim1,\] so the integral in (\ref{ABC-condn}) converges if and only if \[\int_a^tf'(x)\,\mathrm{d}x\] converges, for which it suffices that $f$ is differentiable and $f'$ is $L^1$.
\end{proof}

Lemma \ref{AB-fn-spaces} shows how we chose the conditions for the function $f$ in Definitions \ref{ABRL-defn} and \ref{ABC-defn}: it is on these function spaces for $f$ that the ABR and ABC fractional derivatives are well-defined. (It is shown in \cite{samko} that the RL integral is well-defined for $f\in L^1[a,b]$.)

We now prove the main result of this section.

\begin{thm}
\label{ABRL-series}
The ABR fractional derivative can be expressed as
\begin{equation}
\label{ABRL-series-formula1}
\prescript{ABR}{}D^{\alpha}_{a+}f(t)=\frac{B(\alpha)}{1-\alpha}\sum_{n=0}^{\infty}\Big(\frac{-\alpha}{1-\alpha}\Big)^n\frac{\mathrm{d}}{\mathrm{d}t}\Big(\prescript{RL}{}I^{\alpha n+1}_{a+}f(t)\Big),
\end{equation}
or equivalently (provided composition of RL differintegrals of $f$ works as it should) as
\begin{equation}
\label{ABRL-series-formula2}
\prescript{ABR}{}D^{\alpha}_{a+}f(t)=\frac{B(\alpha)}{1-\alpha}\sum_{n=0}^{\infty}\Big(\frac{-\alpha}{1-\alpha}\Big)^n\prescript{RL}{}I^{\alpha n}_{a+}f(t),
\end{equation}
each series converging locally uniformly in $t$ for any $a,\alpha,f$ satisfying the conditions laid out in Definition \ref{ABRL-defn}.
\end{thm}

\begin{proof}
The Mittag-Leffler function $E_{\alpha}(x)$ is an entire function of $x$, the series (\ref{ML}) converging locally uniformly in the whole complex plane. So the ABR fractional derivative can be rewritten as follows:
\begin{align*}
\prescript{ABR}{}D^{\alpha}_{a+}f(t)&=\frac{B(\alpha)}{1-\alpha}\frac{\mathrm{d}}{\mathrm{d}t}\int_a^tf(x)\sum_{n=0}^{\infty}\frac{(-\alpha)^n(t-x)^{\alpha n}}{(1-\alpha)^n\Gamma(\alpha n+1)}\,\mathrm{d}x \\
&=\frac{B(\alpha)}{1-\alpha}\sum_{n=0}^{\infty}\Big(\frac{-\alpha}{1-\alpha}\Big)^n\frac{1}{\Gamma(\alpha n+1)}\frac{\mathrm{d}}{\mathrm{d}t}\int_a^t(t-x)^{\alpha n}f(x)\,\mathrm{d}x \\
&=\frac{B(\alpha)}{1-\alpha}\sum_{n=0}^{\infty}\Big(\frac{-\alpha}{1-\alpha}\Big)^n\frac{\mathrm{d}}{\mathrm{d}t}\Big(\prescript{RL}{}I^{\alpha n+1}_{a+}f(t)\Big),
\end{align*}
where $\prescript{RL}{}I$ is the standard Riemann--Liouville fractional integral defined in section \ref{sec-intro}.
\end{proof}

In fractional calculus, dealing with convergent series is a standard requirement. Let us recall that the Grunwald-Letnikov differintegral is expressed as a series, and so are the expressions of the fractional Leibniz rule and fractional chain rule. In fact these fundamental tools and formulae reflect the non-locality of the operators and play a genuine role in many real world applications. Thus we expect that the series formula given by Theorem \ref{ABRL-series} may be more applicable than the original formula given by Definition \ref{ABRL-defn}, in many contexts and applications where the non-locality property is important.

The new series formula may also be useful from a numerical point of view. The original formula for AB derivatives is written in terms of the transcendental Mittag-Leffler function, and any explicit calculations of AB derivatives would necessarily involve dealing with this function in some way. But with the new formula, we can obtain an approximation to the AB derivative by truncating the series after some finite number of terms and then using standard Riemann--Liouville numerical methods to estimate each term of the sum. We have not investigated the numerical applications in detail, but we hope that the results of the current work will make it easier to do so.

As we shall see later on in sections \ref{sec-product} and \ref{sec-chain}, the series formula also enables us to prove fundamental results such as the product rule and chain rule for fractional AB derivatives.

Theorem \ref{ABRL-series} enables us to derive the formula for Laplace transforms of ABR derivatives in a way different from that used in \cite{atangana}. The following identity is equation (9) in \cite{atangana}:
\begin{equation}
\label{ABRL-Laplace-old}
\widehat{\prescript{ABR}{}D^{\alpha}_{0+}f}(s)=\frac{B(\alpha)}{1-\alpha}\bigg(\frac{s^{\alpha}}{s^{\alpha}+\tfrac{\alpha}{1-\alpha}}\bigg)\hat{f}(s),
\end{equation}
valid for any sufficiently well-behaved function $f$. We can now prove this  by taking Laplace transforms directly on the series in (\ref{ABRL-series-formula1}):
\begin{align*}
\widehat{\prescript{ABR}{}D^{\alpha}_{0+}f}(s)&=\frac{B(\alpha)}{1-\alpha}\sum_{n=0}^{\infty}\Big(\frac{-\alpha}{1-\alpha}\Big)^n\widehat{\bigg(\frac{\mathrm{d}}{\mathrm{d}t}\prescript{RL}{}I^{\alpha n+1}_{0+}f\bigg)}(s) \\
&=\frac{B(\alpha)}{1-\alpha}\sum_{n=0}^{\infty}\Big(\frac{-\alpha}{1-\alpha}\Big)^n\bigg(s\big(s^{-\alpha n-1}\hat{f}(s)\big)-\prescript{RL}{}I^{\alpha n+1}_{0+}f(0)\bigg) \\
&=\frac{B(\alpha)}{1-\alpha}\sum_{n=0}^{\infty}\bigg[\Big(\tfrac{-\alpha}{1-\alpha}s^{-\alpha}\Big)^n\bigg]\hat{f}(s)-\frac{B(\alpha)}{1-\alpha}\sum_{n=0}^{\infty}\Big(\frac{-\alpha}{1-\alpha}\Big)^n\prescript{RL}{}I^{\alpha n+1}_{0+}f(0) \\
&=\frac{B(\alpha)}{1-\alpha}\Big(1-\tfrac{-\alpha}{1-\alpha}s^{-\alpha}\Big)^{-1}\hat{f}(s)-\prescript{ABR}{}D^{\alpha}_{0+}\circ\prescript{RL}{}I_{0+}f(0) \\
&=\frac{B(\alpha)}{1-\alpha}\bigg(\frac{s^{\alpha}}{s^{\alpha}+\tfrac{\alpha}{1-\alpha}}\bigg)\hat{f}(s)-\prescript{ABR}{}D^{\alpha}_{0+}\circ\prescript{RL}{}I_{0+}f(0).
\end{align*}
Thus we have derived the following more general formula for Laplace transforms of ABR derivatives, which reduces to (\ref{ABRL-Laplace-old}) when $f$ has sufficiently nice convergence properties at the lower limit $t=0$:
\begin{equation}
\label{ABRL-Laplace-new}
\widehat{\prescript{ABR}{}D^{\alpha}_{0+}f}(s)=\frac{B(\alpha)}{1-\alpha}\bigg(\frac{s^{\alpha}}{s^{\alpha}+\tfrac{\alpha}{1-\alpha}}\bigg)\hat{f}(s)-\prescript{ABR}{}D^{\alpha}_{0+}\circ\prescript{RL}{}I_{0+}f(0).
\end{equation}

Theorem \ref{ABRL-series} also has a number of other useful corollaries.

\begin{coroll}
\label{ABRL-inverse}
The AB fractional integral operator $\prescript{AB}{}I^{\alpha}_{a+},$ defined by
\begin{equation}
\label{ABI}
\prescript{AB}{}I^{\alpha}_{a+}f(t)=\frac{1-\alpha}{B(\alpha)}f(t)+\frac{\alpha}{B(\alpha)}\prescript{RL}{}I_{a+}^{\alpha}f(t),
\end{equation}
is both a left and right inverse to the ABR fractional differential operator $\prescript{ABR}{}D^{\alpha}_{a+}$ whenever $a,\alpha,f$ satisfy the conditions from Definition \ref{ABRL-defn}.
\end{coroll}

\begin{proof}
First let us note that the expression (\ref{ABI}) is well-defined if and only if the RL integral $\prescript{RL}{}I_{a+}^{\alpha}f(t)$ is well-defined, which matches with our assumptions on $f$ for the $\alpha$th ABR fractional derivative to be well-defined.

The expression (\ref{ABRL-series-formula2}) can be \textit{formally} rewritten as
\begin{equation*}
\prescript{ABR}{}D^{\alpha}_{a+}f(t)=\frac{B(\alpha)}{1-\alpha}\bigg(1-\big(\tfrac{-\alpha}{1-\alpha}\big)\prescript{RL}{}I^{\alpha}_{a+}\bigg)^{-1}f(t)=\bigg(\frac{1-\alpha}{B(\alpha)}+\frac{\alpha}{B(\alpha)}\prescript{RL}{}I^{\alpha}_{a+}\bigg)^{-1}f(t),
\end{equation*}
which gives us the motivation for using (\ref{ABI}) as our definition for AB fractional integrals. Now we need to prove rigorously that
\begin{equation}
\label{ABRL-ABI-inverse}
\prescript{ABR}{}D^{\alpha}_{a+}\Big(\prescript{AB}{}I^{\alpha}_{a+}f(t)\Big)=f(t)
\end{equation}
and
\begin{equation}
\label{ABI-ABRL-inverse}
\prescript{AB}{}I^{\alpha}_{a+}\Big(\prescript{ABR}{}D^{\alpha}_{a+}f(t)\Big)=f(t)
\end{equation}
for $a,\alpha,f$ as stated. We proceed as follows.
\begin{align*}
\prescript{ABR}{}D^{\alpha}_{a+}\Big(\prescript{AB}{}I^{\alpha}_{a+}f(t)\Big)&=\prescript{ABR}{}D^{\alpha}_{a+}\Big(\tfrac{1-\alpha}{B(\alpha)}f(t)+\tfrac{\alpha}{B(\alpha)}\prescript{RL}{}I_{a+}^{\alpha}f(t)\Big) \\
&=\frac{1-\alpha}{B(\alpha)}\prescript{ABR}{}D^{\alpha}_{a+}f(t)+\frac{\alpha}{B(\alpha)}\prescript{ABR}{}D^{\alpha}_{a+}\Big(\prescript{RL}{}I_{a+}^{\alpha}f(t)\Big) \\
&=\sum_{n=0}^{\infty}\Big(\frac{-\alpha}{1-\alpha}\Big)^n\prescript{RL}{}I^{\alpha n}_{a+}f(t)+\frac{\alpha}{1-\alpha}\sum_{n=0}^{\infty}\Big(\frac{-\alpha}{1-\alpha}\Big)^n\prescript{RL}{}I^{\alpha n}_{a+}\Big(\prescript{RL}{}I_{a+}^{\alpha}f(t)\Big) \\
&=\sum_{n=0}^{\infty}\Big(\frac{-\alpha}{1-\alpha}\Big)^n\prescript{RL}{}I^{\alpha n}_{a+}f(t)-\sum_{n=0}^{\infty}\Big(\frac{-\alpha}{1-\alpha}\Big)^{n+1}\prescript{RL}{}I^{\alpha n+\alpha}_{a+}f(t) \\
&=f(t); \\
\prescript{AB}{}I^{\alpha}_{a+}\Big(\prescript{ABR}{}D^{\alpha}_{a+}f(t)\Big)&=\frac{1-\alpha}{B(\alpha)}\prescript{ABR}{}D^{\alpha}_{a+}f(t)+\frac{\alpha}{B(\alpha)}\prescript{RL}{}I_{a+}^{\alpha}\Big(\prescript{ABR}{}D^{\alpha}_{a+}f(t)\Big) \\
&=\sum_{n=0}^{\infty}\Big(\frac{-\alpha}{1-\alpha}\Big)^n\prescript{RL}{}I^{\alpha n}_{a+}f(t)+\frac{\alpha}{1-\alpha}\prescript{RL}{}I_{a+}^{\alpha}\bigg(\sum_{n=0}^{\infty}\Big(\frac{-\alpha}{1-\alpha}\Big)^n\prescript{RL}{}I^{\alpha n}_{a+}f(t)\bigg) \\
&=\sum_{n=0}^{\infty}\Big(\frac{-\alpha}{1-\alpha}\Big)^n\prescript{RL}{}I^{\alpha n}_{a+}f(t)-\sum_{n=0}^{\infty}\Big(\frac{-\alpha}{1-\alpha}\Big)^{n+1}\prescript{RL}{}I^{\alpha+\alpha n}_{a+}f(t) \\
&=f(t),
\end{align*}
where in both cases we have used the fact that Riemann--Liouville fractional integrals have nice composition properties: the RL integral of an RL integral is an RL integral of the appropriate order (see for instance Lemma 2.3 of \cite{kilbas}). Note that since $f\in L^1(a,b)$, all RL fractional integrals of $f$ are well-defined, and therefore so are the ABR derivative of $f$ (by Lemma \ref{AB-fn-spaces}), the AB integral of $f$ (by the definition (\ref{ABI})), and their compositions (by Lemma 2.3 of \cite{kilbas} again).
\end{proof}

\begin{example}
As an example to verify the result of Corollary \ref{ABRL-inverse}, we set $f(t)=1$ and calculate $\prescript{AB}{}I^{\alpha}_{a+}\Big(\prescript{ABR}{}D^{\alpha}_{a+}f(t)\Big)$ explicitly using the definitions (\ref{ABRL}) and (\ref{ABI}). Firstly, by (\ref{ABRL}) we have
\begin{align*}
\prescript{ABR}{}D^{\alpha}_{a+}(1)&=\frac{B(\alpha)}{1-\alpha}\frac{\mathrm{d}}{\mathrm{d}t}\int_a^tE_\alpha\Big(\tfrac{-\alpha}{1-\alpha}(t-x)^{\alpha}\Big)\,\mathrm{d}x \\
&=\frac{B(\alpha)}{1-\alpha}\frac{\mathrm{d}}{\mathrm{d}t}\sum_{n=0}^{\infty}\frac{\big(\tfrac{-\alpha}{1-\alpha}\big)^n}{\Gamma(n\alpha+1)}\int_a^t(t-x)^{n\alpha}\,\mathrm{d}x \\
&=\frac{B(\alpha)}{1-\alpha}\frac{\mathrm{d}}{\mathrm{d}t}\Bigg((t-a)\sum_{n=0}^{\infty}\frac{\big[\tfrac{-\alpha}{1-\alpha}(t-a)^{\alpha}\big]^n}{\Gamma(n\alpha+2)}\Bigg) \\
&=\tfrac{B(\alpha)}{1-\alpha}E_{\alpha}\Big(\tfrac{-\alpha}{1-\alpha}(t-a)^{\alpha}\Big).
\end{align*}
Then by applying (\ref{ABI}) to this, we have
\begin{align*}
\prescript{AB}{}I^{\alpha}_{a+}\Big(\prescript{ABR}{}D^{\alpha}_{a+}(1)\Big)&=\frac{1-\alpha}{B(\alpha)}\Big(\prescript{ABR}{}D^{\alpha}_{a+}(1)\Big)+\frac{\alpha}{B(\alpha)}\prescript{RL}{}I_{a+}^{\alpha}\Big(\prescript{ABR}{}D^{\alpha}_{a+}(1)\Big) \\
&=E_{\alpha}\Big(\tfrac{-\alpha}{1-\alpha}(t-a)^{\alpha}\Big)+\frac{\alpha}{B(\alpha)\Gamma(\alpha)}\int_a^t(t-x)^{\alpha-1}\tfrac{B(\alpha)}{1-\alpha}E_{\alpha}\Big(\tfrac{-\alpha}{1-\alpha}(x-a)^{\alpha}\Big)\,\mathrm{d}x \\
&=E_{\alpha}\Big(\tfrac{-\alpha}{1-\alpha}(t-a)^{\alpha}\Big)+\frac{\alpha}{(1-\alpha)\Gamma(\alpha)}\sum_{n=0}^{\infty}\frac{\big(\tfrac{-\alpha}{1-\alpha}\big)^n}{\Gamma(n\alpha+1)}(t-a)^{(n+1)\alpha}B(\alpha,n\alpha+1) \\
&=E_{\alpha}\Big(\tfrac{-\alpha}{1-\alpha}(t-a)^{\alpha}\Big)-\sum_{n=0}^{\infty}\frac{\big(\tfrac{-\alpha}{1-\alpha}\big)^{n+1}(t-a)^{(n+1)\alpha}}{\Gamma((n+1)\alpha+1)} \\
&=E_{\alpha}\Big(\tfrac{-\alpha}{1-\alpha}(t-a)^{\alpha}\Big)-\Big[E_{\alpha}\Big(\tfrac{-\alpha}{1-\alpha}(t-a)^{\alpha}\Big)-1\Big] \\
&=1,
\end{align*}
precisely as expected.
\end{example}
Note that the equation (\ref{ABI-ABRL-inverse}) should no longer be valid when $\alpha=1$, because in this case the result for ordinary derivatives would be a Newton--Leibniz rule rather than a direct inverse relation. However, we can see from examining the above example that the $\alpha$th ABR derivative does not always converge to the standard 1st derivative as $\alpha\rightarrow1$. Specifically,
\begin{align*}
\lim_{\alpha\rightarrow1}\Big(\prescript{ABR}{}D^{\alpha}_{a+}(1)\Big)&=\lim_{\alpha\rightarrow1}\Big(\tfrac{B(\alpha)}{1-\alpha}E_{\alpha}\Big(\tfrac{-\alpha}{1-\alpha}(t-a)^{\alpha}\Big)\Big) \\
&=\lim_{\alpha\rightarrow1}\Big(\tfrac{1}{1-\alpha}\exp\Big(\tfrac{-(t-a)}{1-\alpha}\Big)\Big)=\delta(t-a),
\end{align*}
the Dirac delta function. This is equal to zero (the 1st derivative of $f(t)$ in this case) almost everywhere, but the blowup at the limiting point $t=a$ changes the behaviour of the integral. In fact, we can observe the same behaviour with the classical Riemann--Liouville derivative: in that model, the $\alpha$th derivative of $f(t)=1$ is $\frac{(t-a)^{-\alpha}}{\Gamma(1-\alpha)}$, which again blows up at the limiting point $t=a$. In both models, the double limit as $\alpha\rightarrow1$ and $t\rightarrow a$ needs to be handled with care. For the AB model, as $\alpha\rightarrow1$, the ABR derivative converges to the standard derivative only \textit{almost} everywhere (everywhere except $t=a$), and thus the Newton--Leibniz formula at $\alpha=1$ cannot be derived as a limit of the corresponding result (\ref{ABI-ABRL-inverse}) for $\alpha<1$.

In this way, we can understand the apparent discrepancy between the direct inverse relationship of the ABR derivative and AB integral and the Newton--Leibniz relationship of the standard 1st-order derivative and integral. See, however, Corollary \ref{ABC-NL} below for a valid Newton--Leibniz relationship between the ABC derivative and the AB integral.

\begin{coroll}
The AB integral operators and ABR differential operators form a commutative family of differintegral operators:
\begin{align*}
\prescript{ABR}{}D^{\alpha}_{a+}\Big(\prescript{ABR}{}D^{\beta}_{a+}f(t)\Big)&=\prescript{ABR}{}D^{\beta}_{a+}\Big(\prescript{ABR}{}D^{\alpha}_{a+}f(t)\Big) \\
\prescript{AB}{}I^{\alpha}_{a+}\Big(\prescript{AB}{}I^{\beta}_{a+}f(t)\Big)&=\prescript{AB}{}I^{\beta}_{a+}\Big(\prescript{AB}{}I^{\alpha}_{a+}f(t)\Big) \\
\prescript{ABR}{}D^{\alpha}_{a+}\Big(\prescript{AB}{}I^{\beta}_{a+}f(t)\Big)&=\prescript{AB}{}I^{\beta}_{a+}\Big(\prescript{ABR}{}D^{\alpha}_{a+}f(t)\Big) \\
\end{align*}
for $\alpha,\beta\in(0,1)$ and $a,f$ satisfying the conditions from Definition \ref{ABRL-defn}.
\end{coroll}

\begin{proof}
For the first identity, we use equation (\ref{ABRL-series-formula2}) to get
\begin{align*}
\prescript{ABR}{}D^{\alpha}_{a+}\Big(\prescript{ABR}{}D^{\beta}_{a+}f(t)\Big)&=\frac{B(\alpha)}{1-\alpha}\sum_{n=0}^{\infty}\Big(\frac{-\alpha}{1-\alpha}\Big)^n\prescript{RL}{}I^{\alpha n}_{a+}\Bigg[\frac{B(\beta)}{1-\beta}\sum_{m=0}^{\infty}\Big(\frac{-\beta}{1-\beta}\Big)^m\prescript{RL}{}I^{\beta m}_{a+}f(t)\Bigg] \\
&=\frac{B(\alpha)B(\beta)}{(1-\alpha)(1-\beta)}\sum_{m,n}\frac{(-\alpha)^n(-\beta)^m}{(1-\alpha)^n(1-\beta)^m}\prescript{RL}{}I^{\alpha n+\beta m}_{a+}f(t).
\end{align*}
This expression is symmetric in $\alpha$ and $\beta$ -- it remains the same if these two variables are swapped -- so it must be equal to $\prescript{ABR}{}D^{\beta}_{a+}\big(\prescript{ABR}{}D^{\alpha}_{a+}f(t)\big)$, as required.

For the second identity, we use equation (\ref{ABI}) to get
\begin{equation*}
\prescript{AB}{}I^{\alpha}_{a+}\Big(\prescript{AB}{}I^{\beta}_{a+}f(t)\Big)=\bigg(\frac{1-\alpha}{B(\alpha)}+\frac{\alpha}{B(\alpha)}\prescript{RL}{}I^{\alpha}_{a+}\bigg)\bigg(\frac{1-\beta}{B(\beta)}+\frac{\beta}{B(\beta)}\prescript{RL}{}I^{\beta}_{a+}\bigg)f(t),
\end{equation*}
which again is symmetric in $\alpha$ and $\beta$, since the Riemann--Liouville fractional integral operators commute (this is another consequence of Lemma 2.3 in \cite{kilbas}).

For the third identity, we use equations (\ref{ABRL-series-formula2}) and (\ref{ABI}) together:
\begin{align*}
\prescript{ABR}{}D^{\alpha}_{a+}\Big(\prescript{AB}{}I^{\beta}_{a+}f(t)\Big)&=\frac{B(\alpha)}{1-\alpha}\sum_{n=0}^{\infty}\Big(\frac{-\alpha}{1-\alpha}\Big)^n\prescript{RL}{}I^{\alpha n}_{a+}\Bigg[\bigg(\frac{1-\beta}{B(\beta)}+\frac{\beta}{B(\beta)}\prescript{RL}{}I^{\beta}_{a+}\bigg)f(t)\Bigg] \\
&=\bigg(\frac{1-\beta}{B(\beta)}+\frac{\beta}{B(\beta)}\prescript{RL}{}I^{\beta}_{a+}\bigg)\Bigg[\frac{B(\alpha)}{1-\alpha}\sum_{n=0}^{\infty}\Big(\frac{-\alpha}{1-\alpha}\Big)^n\prescript{RL}{}I^{\alpha n}_{a+}f(t)\Bigg] \\
&=\prescript{AB}{}I^{\beta}_{a+}\Big(\prescript{ABR}{}D^{\alpha}_{a+}f(t)\Big),
\end{align*}
again using the fact that Riemann--Liouville fractional integral operators are commutative, as well as the local uniform convergence of the series in (\ref{ABRL-series-formula2}).
\end{proof}

Note that both of the above corollaries could also be proved using the Laplace transform formula (\ref{ABRL-Laplace-old}) which was established in \cite{atangana}. But the advantage of the new approach, proving them directly from Theorem \ref{ABRL-series}, is that it works for \textit{all} functions $f$ such that the AB fractional derivatives and integrals are well-defined, not just those $f$ which have well-defined Laplace transforms. The proofs are more direct, without the need to pass back and forth between the time domain and the frequency domain.

The same method used to prove Theorem \ref{ABRL-series} also works to establish the following analogous result for ABC fractional derivatives.

\begin{thm}
\label{ABC-series}
The ABC fractional derivative can be expressed as
\begin{equation}
\label{ABC-series-formula}
\prescript{ABC}{}D^{\alpha}_{a+}f(t)=\frac{B(\alpha)}{1-\alpha}\sum_{n=0}^{\infty}\Big(\frac{-\alpha}{1-\alpha}\Big)^n\prescript{RL}{}I^{\alpha n+1}_{a+}f'(t),
\end{equation}
this series converging locally uniformly in $t$ for any $a,\alpha,f$ satisfying the conditions from Definition \ref{ABC-defn}.
\end{thm}

And just as we did above for ABR derivatives, we can use this result to derive the following expression for Laplace transforms of ABC derivatives:
\begin{equation}
\label{ABC-Laplace}
\widehat{\prescript{ABC}{}D^{\alpha}_{a+}f}(s)=\frac{B(\alpha)}{1-\alpha}\cdot\frac{s^{\alpha-1}}{s^{\alpha}+\tfrac{\alpha}{1-\alpha}}\Big(s\hat{f}(s)-f(0)\Big)
\end{equation}
for any function $f$ with well-defined Laplace transform $\hat{f}$ and ABC fractional derivative $\prescript{ABC}{}D^{\alpha}_{a+}f(t)$. This was equation (10) in \cite{atangana}.

The following result was already shown in \cite{djida} using Laplace transforms, but we can now prove it in a much more elementary way, with fewer required assumptions on the function $f$, by using the series formula from Theorem \ref{ABC-series}.
\begin{coroll}
\label{ABC-NL}
The AB fractional integral and the ABC fractional derivative satisfy the following Newton--Leibniz formula:
\begin{equation}
\label{ABC-NL-formula}
\prescript{AB}{}I^{\alpha}_{a+}\Big(\prescript{ABC}{}D^{\alpha}_{a+}f(t)\Big)=f(t)-f(a),
\end{equation}
valid whenever $a,\alpha,f$ satisfy the conditions from Definition \ref{ABC-defn}.
\end{coroll}

\begin{proof}
By the definition (\ref{ABI}) of the AB integral, together with the series formula (\ref{ABC-series-formula}) for the ABC derivative, we have
\begin{align*}
\prescript{AB}{}I^{\alpha}_{a+}\Big(\prescript{ABC}{}D^{\alpha}_{a+}f(t)\Big)&=\frac{1-\alpha}{B(\alpha)}\prescript{ABC}{}D^{\alpha}_{a+}f(t)+\frac{\alpha}{B(\alpha)}\prescript{RL}{}I_{a+}^{\alpha}\Big(\prescript{ABC}{}D^{\alpha}_{a+}f(t)\Big) \\
&=\sum_{n=0}^{\infty}\big(\tfrac{-\alpha}{1-\alpha}\big)^n\prescript{RL}{}I^{\alpha n+1}_{a+}f'(t)+\frac{\alpha}{1-\alpha}\prescript{RL}{}I_{a+}^{\alpha}\bigg(\sum_{n=0}^{\infty}\big(\tfrac{-\alpha}{1-\alpha}\big)^n\prescript{RL}{}I^{\alpha n+1}_{a+}f'(t)\bigg) \\
&=\sum_{n=0}^{\infty}\big(\tfrac{-\alpha}{1-\alpha}\big)^n\prescript{RL}{}I^{\alpha n+1}_{a+}f'(t)-\sum_{n=0}^{\infty}\big(\tfrac{-\alpha}{1-\alpha}\big)^{n+1}\prescript{RL}{}I^{\alpha+\alpha n+1}_{a+}f'(t) \\
&=\prescript{RL}{}I^{1}_{a+}f'(t)=f(t)-f(a),
\end{align*}
by the standard Newton--Leibniz formula.
\end{proof}

This is significant because the Newton--Leibniz formula is a required element in the derivation of a theory of fractional vector calculus \cite{tarasov2}. Thus, knowing that a Newton--Leibniz analogue holds for ABC derivatives may enable us to construct a theory of fractional vector calculus in the AB model too.

\section{Some ordinary differential equations}
\label{sec-ode}

\subsection{Linear ODEs of Riemann--Liouville type}
As a basic first case, consider the following simple fractional ODE:
\begin{equation}
\label{ODE1}
\prescript{ABR}{}D^{\alpha}_{0+}f(t)-\tfrac{B(\alpha)}{1-\alpha}f(t)=g(t)
\end{equation}
where $f$ and $g$ are Laplace-transformable functions and $\alpha\in(0,1)$. Taking Laplace transforms of (\ref{ODE1}), using the formula (\ref{ABRL-Laplace-new}), yields $$\tfrac{B(\alpha)}{1-\alpha}\bigg(\frac{s^{\alpha}}{s^{\alpha}+\tfrac{\alpha}{1-\alpha}}\bigg)\hat{f}(s)-\prescript{ABR}{}D^{\alpha}_{0+}\circ\prescript{RL}{}I_{0+}f(0)-\tfrac{B(\alpha)}{1-\alpha}\hat{f}(s)=\hat{g}(s),$$ and therefore $$\tfrac{B(\alpha)}{1-\alpha}\bigg(\frac{\tfrac{-\alpha}{1-\alpha}}{s^{\alpha}+\tfrac{\alpha}{1-\alpha}}\bigg)\hat{f}(s)=\hat{g}(s)+\prescript{ABR}{}D^{\alpha}_{0+}\circ\prescript{RL}{}I_{0+}f(0).$$ Thus the Laplace-transformed solution $\hat{f}(s)$ is given by $$\hat{f}(s)=\tfrac{-(1-\alpha)^2}{\alpha B(\alpha)}\Big(s^{\alpha}+\tfrac{\alpha}{1-\alpha}\Big)\hat{h}(s),$$ where the function $\hat{h}(s)$ is defined by $$\hat{h}(s)=\hat{g}(s)+\prescript{ABR}{}D^{\alpha}_{0+}\circ\prescript{RL}{}I_{0+}f(0)$$ or equivalently
\begin{equation}
\label{ODE1-h-defn}
h(t)=g(t)+\Big[\prescript{ABR}{}D^{\alpha}_{0+}\circ\prescript{RL}{}I_{0+}f(0)\Big]\delta(t).
\end{equation}

Now the Laplace transform of the Riemann--Liouville fractional derivative $\prescript{RL}{}D^{\alpha}_{0+}h(t)$ is $s^{\alpha}\hat{h}(s)-\prescript{RL}{}I^{1-\alpha}_{0+}h(0)$, so the Laplace transform of $\prescript{RL}{}D^{\alpha}_{0+}h(t)+\prescript{RL}{}I^{1-\alpha}_{0+}h(0)\delta(t)$ is simply $s^{\alpha}\hat{h}(s)$. Thus the unique Laplace-transformable solution to equation (\ref{ODE1}) is
\begin{equation*}
f(t)=\tfrac{-(1-\alpha)^2}{\alpha B(\alpha)}\Big(\prescript{RL}{}D^{\alpha}_{0+}h(t)+\prescript{RL}{}I^{1-\alpha}_{0+}h(0)\delta(t)\Big)-\tfrac{1-\alpha}{B(\alpha)}h(t),
\end{equation*}
where $h$ is defined by (\ref{ODE1-h-defn}) and therefore depends only -- and linearly -- on $g$ and initial values of $f$.

More generally, let us consider the following fractional ODE, inhomogeneous with arbitrary constant coefficients:
\begin{equation}
\label{ODE2}
\prescript{ABR}{}D^{\alpha}_{0+}f(t)-Af(t)=g(t)
\end{equation}
where $f$ and $g$ are Laplace-transformable functions, $\alpha\in(0,1)$, and $A$ is a constant. Write $k\coloneqq\tfrac{1-\alpha}{B(\alpha)}A$ for ease of notation, so that the equation (\ref{ODE2}) becomes $$\prescript{ABR}{}D^{\alpha}_{0+}f(t)-\tfrac{B(\alpha)}{1-\alpha}kf(t)=g(t).$$ Then take Laplace transforms of this equation, using the formula (\ref{ABRL-Laplace-new}), to get $$\frac{B(\alpha)}{1-\alpha}\Bigg(\frac{s^{\alpha}-k[s^{\alpha}+\tfrac{\alpha}{1-\alpha}]}{s^{\alpha}+\tfrac{\alpha}{1-\alpha}}\Bigg)\hat{f}(s)-\prescript{ABR}{}D^{\alpha}_{0+}\circ\prescript{RL}{}I_{0+}f(0)=\hat{g}(s).$$ Thus the Laplace-transformed solution $\hat{f}(s)$ can be written as:
\begin{align*}
\hat{f}(s)&=\frac{1-\alpha}{B(\alpha)}\Big(s^{\alpha}+\tfrac{\alpha}{1-\alpha}\Big)\Big((1-k)s^{\alpha}-\tfrac{k\alpha}{1-\alpha}\Big)^{-1}\Big(\hat{g}(s)+\prescript{ABR}{}D^{\alpha}_{0+}\circ\prescript{RL}{}I_{0+}f(0)\Big) \\
&=\frac{1-\alpha}{(1-k)B(\alpha)}\Big(1+\tfrac{\alpha}{1-\alpha}s^{-\alpha}\Big)\Big(1-\tfrac{k\alpha}{(1-k)(1-\alpha)}s^{-\alpha}\Big)^{-1}\hat{h}(s) \\
&=\frac{1-\alpha}{(1-k)B(\alpha)}\Big(1+\tfrac{k\alpha}{(1-k)(1-\alpha)}s^{-\alpha}+\tfrac{k^2\alpha^2}{(1-k)^2(1-\alpha)^2}s^{-2\alpha}+\dots\Big)\Big(1+\tfrac{\alpha}{1-\alpha}s^{-\alpha}\Big)\hat{h}(s),
\end{align*}
where $h$ is defined by (\ref{ODE1-h-defn}) as before.

From here, a simple way of proceeding yields
\begin{equation*}
\hat{f}(s)=\tfrac{1}{k}\Big(1+\tfrac{k\alpha}{(1-k)(1-\alpha)}s^{-\alpha}+\tfrac{k^2\alpha^2}{(1-k)^2(1-\alpha)^2}s^{-2\alpha}+\dots\Big)\widehat{\prescript{AB}{}I^{\alpha}_{0+}h}(s)
\end{equation*}
and therefore
\begin{align}
\nonumber f(t)&=\frac{1}{1-k}\frac{\mathrm{d}}{\mathrm{d}t}\int_0^tE_{\alpha}\Big(\tfrac{k\alpha}{(1-k)(1-\alpha)}(t-x)^{\alpha}\Big)\prescript{AB}{}I^{\alpha}_{0+}h(x)\,\mathrm{d}x \\
\label{ODE2-soln1}
&=\frac{1}{1-k}\prescript{AB}{}I^{\alpha}_{0+}h(t)+\frac{1}{1-k}\int_0^t\frac{\mathrm{d}}{\mathrm{d}t}E_{\alpha}\Big(\tfrac{k\alpha}{(1-k)(1-\alpha)}(t-x)^{\alpha}\Big)\prescript{AB}{}I^{\alpha}_{0+}h(x)\,\mathrm{d}x.
\end{align}
Alternatively, a slightly more interesting way of proceeding yields
\begin{align*}
\hat{f}(s)&=\frac{1-\alpha}{(1-k)B(\alpha)}\bigg(1+\Big[\tfrac{\alpha}{1-\alpha}+\tfrac{k\alpha}{(1-k)(1-\alpha)}\Big]s^{-\alpha}+\Big[\tfrac{k\alpha^2}{(1-k)(1-\alpha)^2}+\tfrac{k^2\alpha^2}{(1-k)^2(1-\alpha)^2}\Big]s^{-2\alpha}+\dots\bigg)\hat{h}(s) \\
\nonumber &=\frac{1-\alpha}{(1-k)B(\alpha)}\Big(1+\tfrac{\alpha}{(1-k)(1-\alpha)}s^{-\alpha}+\tfrac{k\alpha^2}{(1-k)^2(1-\alpha)^2}s^{-2\alpha}+\dots\Big)\hat{h}(s) \\
\nonumber &=\frac{1-\alpha}{k(1-k)B(\alpha)}\Big((k-1)+1+\tfrac{k\alpha}{(1-k)(1-\alpha)}s^{-\alpha}+\tfrac{k^2\alpha^2}{(1-k)^2(1-\alpha)^2}s^{-2\alpha}+\dots\Big)\hat{h}(s)
\end{align*}
and therefore
\begin{align}
\nonumber f(t)&=\frac{1-\alpha}{-kB(\alpha)}h(t)+\frac{1-\alpha}{k(1-k)B(\alpha)}\frac{\mathrm{d}}{\mathrm{d}t}\int_0^tE_{\alpha}\Big(\tfrac{k\alpha}{(1-k)(1-\alpha)}(t-x)^{\alpha}\Big)h(x)\,\mathrm{d}x \\
\nonumber &=-\frac{1}{A}h(t)+\frac{1}{A(1-k)}\Bigg[h(t)+\int_0^t\frac{\mathrm{d}}{\mathrm{d}t}E_{\alpha}\Big(\tfrac{k\alpha}{(1-k)(1-\alpha)}(t-x)^{\alpha}\Big)h(x)\,\mathrm{d}x\Bigg] \\
\label{ODE2-soln2}
&=\frac{k}{A(1-k)}h(t)+\frac{1}{A(1-k)}\int_0^t\frac{\mathrm{d}}{\mathrm{d}t}E_{\alpha}\Big(\tfrac{k\alpha}{(1-k)(1-\alpha)}(t-x)^{\alpha}\Big)h(x)\,\mathrm{d}x.
\end{align}
So the (equivalent) expressions (\ref{ODE2-soln1}) and (\ref{ODE2-soln2}), where the function $h$ is defined by (\ref{ODE1-h-defn}), provide the unique Laplace-transformable solution to the ODE (\ref{ODE2}).

By applying this argument multiple times, we can solve any linear sequential fractional ODE of the form
\begin{equation}
\label{ODE3}
\Big(\prescript{ABR}{}D^{\alpha}_{0+}-A\Big)\Big(\prescript{ABR}{}D^{\beta}_{0+}-B\Big)\dots\Big(\prescript{ABR}{}D^{\gamma}_{0+}-C\Big)f(t)=g(t),
\end{equation}
where $g$ is a Laplace-transformable function and $\alpha,\beta,\dots,\gamma\in(0,1)$ and $A,B,\dots,C$ are constants. In each case, we can construct a nested integral formula for the unique Laplace-transformable solution $f$ to the ODE (\ref{ODE3}).

\begin{example}
\label{ODE-example1}
As an example application of this method, let us consider the following sequential fractional ODE:
\begin{equation}
\label{ABR-example}
\prescript{ABR}{}D^{1/2}_{0+}\circ\prescript{ABR}{}D^{1/2}_{0+}f(t)=f(t)+g(t),
\end{equation}
where we take the normalisation function $B(\alpha)$ to be identically $1$. This ODE can be rewritten as \[\Big(\prescript{ABR}{}D^{1/2}_{0+}-1\Big)\Big(\prescript{ABR}{}D^{1/2}_{0+}+1\Big)f(t)=g(t),\] which we can split into a coupled pair of ODEs as follows:
\begin{align}
\label{ABR-example1}
\prescript{ABR}{}D^{1/2}_{0+}j(t)-j(t)&=g(t); \\
\label{ABR-example2}
\prescript{ABR}{}D^{1/2}_{0+}f(t)+f(t)&=j(t).
\end{align}
Note that both (\ref{ABR-example1}) and (\ref{ABR-example2}) are ODEs in the form of (\ref{ODE2}). In the first case, we have $\alpha=1/2$, $A=1$, and therefore $k=1/2$, so the formula (\ref{ODE2-soln2}) becomes \[j(t)=g(t)+2\int_0^t\frac{\mathrm{d}}{\mathrm{d}t}E_{1/2}\big((t-x)^{1/2}\big)g(x)\,\mathrm{d}x.\] In the second case, we have $\alpha=1/2$, $A=-1$, and therefore $k=-1/2$, so the formula (\ref{ODE2-soln2}) becomes \[f(t)=\frac{1}{3}j(t)-\frac{2}{3}\int_0^t\frac{\mathrm{d}}{\mathrm{d}t}E_{1/2}\Big(-\tfrac{1}{3}(t-x)^{1/2}\Big)j(x)\,\mathrm{d}x.\] Thus the solution to the sequential ODE (\ref{ABR-example}) is:
\begin{multline*}
f(t)=\frac{1}{3}g(t)+\frac{2}{3}\int_0^t\frac{\mathrm{d}}{\mathrm{d}t}E_{1/2}\big((t-x)^{1/2}\big)g(x)\,\mathrm{d}x-\frac{2}{3}\int_0^t\frac{\mathrm{d}}{\mathrm{d}t}E_{1/2}\Big(-\tfrac{1}{3}(t-x)^{1/2}\Big)g(x)\,\mathrm{d}x \\ -\frac{2}{3}\int_0^t\frac{\mathrm{d}}{\mathrm{d}t}E_{1/2}\Big(-\tfrac{1}{3}(t-x)^{1/2}\Big)\Bigg[2\int_0^x\frac{\mathrm{d}}{\mathrm{d}x}E_{1/2}\big((x-y)^{1/2}\big)g(y)\,\mathrm{d}y\Bigg]\,\mathrm{d}x.
\end{multline*}
\end{example}

\subsection{Linear ODEs of Caputo type}
Now let us consider the same types of ODEs but with derivatives of Caputo type instead of Riemann--Liouville type. The solution in this case runs in much the same way as before, except that now the initial values are slightly easier to deal with.

As a basic first case, let us consider the following simple fractional ODE, analogous to (\ref{ODE1}):
\begin{equation}
\label{ODE4}
\prescript{ABC}{}D^{\alpha}_{0+}f(t)-\tfrac{B(\alpha)}{1-\alpha}f(t)=g(t)
\end{equation}
where $f$ and $g$ are Laplace-transformable functions and $\alpha\in(0,1)$. Taking Laplace transforms, we get
\begin{align*}
(\ref{ODE1})&\Leftrightarrow\frac{B(\alpha)}{1-\alpha}\bigg(\frac{s^{\alpha-1}}{s^{\alpha}+\tfrac{\alpha}{1-\alpha}}\bigg)\Big(s\hat{f}(s)-f(0)\Big)-\frac{B(\alpha)}{1-\alpha}\hat{f}(s)=\hat{g}(s) \\
&\Leftrightarrow\frac{B(\alpha)}{1-\alpha}\cdot\frac{1}{s^{\alpha}+\tfrac{\alpha}{1-\alpha}}\Big(s^{\alpha}\hat{f}(s)-s^{\alpha-1}f(0)-\big(s^{\alpha}+\tfrac{\alpha}{1-\alpha}\big)\hat{f}(s)\Big)=\hat{g}(s) \\
&\Leftrightarrow\tfrac{B(\alpha)}{1-\alpha}\Big(\tfrac{-\alpha}{1-\alpha}\hat{f}(s)-s^{\alpha-1}f(0)\Big)=\big(s^{\alpha}+\tfrac{\alpha}{1-\alpha}\big)\hat{g}(s) \\
&\Leftrightarrow\tfrac{-\alpha B(\alpha)}{(1-\alpha)^2}\hat{f}(s)=\big(s^{\alpha}+\tfrac{\alpha}{1-\alpha}\big)\hat{g}(s)+\tfrac{B(\alpha)}{1-\alpha}s^{\alpha-1}f(0) \\
&\Leftrightarrow\hat{f}(s)=\tfrac{\alpha-1}{B(\alpha)}\hat{g}(s)-\tfrac{(1-\alpha)^2}{\alpha B(\alpha)}s^{\alpha}\hat{g}(s)+\tfrac{\alpha-1}{\alpha}s^{\alpha-1}f(0) \\
&\Leftrightarrow f(t)=\tfrac{\alpha-1}{B(\alpha)}g(t)-\tfrac{(1-\alpha)^2}{\alpha B(\alpha)}\Big(\prescript{RL}{}D^{\alpha}_{0+}g(t)+\prescript{RL}{}I^{1-\alpha}_{0+}g(0)\delta(t)\Big)+\tfrac{\alpha-1}{\alpha}\cdot\tfrac{t^{-\alpha}}{\Gamma(1-\alpha)}f(0),
\end{align*}
where $\prescript{RL}{}D$ is the Riemann--Liouville fractional derivative. So the unique Laplace-transformable solution to the ODE (\ref{ODE4}) is
\begin{equation*}
f(t)=\tfrac{\alpha-1}{B(\alpha)}g(t)-\tfrac{(1-\alpha)^2}{\alpha B(\alpha)}\prescript{RL}{}D^{\alpha}_{0+}g(t)-\Big(\tfrac{(1-\alpha)^2}{\alpha B(\alpha)}\prescript{RL}{}I^{1-\alpha}_{0+}g(0)\Big)\delta(t)+\tfrac{\alpha-1}{\alpha\Gamma(1-\alpha)}t^{-\alpha}f(0).
\end{equation*}

More generally, let us consider the following fractional ODE, inhomogeneous with arbitrary constant coefficients, identical to (\ref{ODE2}) except with ABC derivatives instead of ABR:
\begin{equation}
\label{ODE5}
\prescript{ABC}{}D^{\alpha}_{0+}f(t)-Af(t)=g(t)
\end{equation}
where $f$ and $g$ are Laplace-transformable functions, $\alpha\in(0,1)$, and $A$ is a constant. Note that this is the AB equivalent of a class of fractional ODEs which has been much studied in the classical Caputo case; see for example \cite{garrappa}. Write $k\coloneqq\tfrac{1-\alpha}{B(\alpha)}A$ again, so that equation (\ref{ODE5}) becomes $$\prescript{ABC}{}D^{\alpha}_{0+}f(t)-\tfrac{B(\alpha)}{1-\alpha}kf(t)=g(t).$$ Then take Laplace transforms of this, to get $$\frac{B(\alpha)}{1-\alpha}\Bigg(\frac{s^{\alpha}\hat{f}(s)-s^{\alpha-1}f(0)-k\big[s^{\alpha}+\tfrac{\alpha}{1-\alpha}\big]\hat{f}(s)}{s^{\alpha}+\tfrac{\alpha}{1-\alpha}}\Bigg)=\hat{g}(s)$$ and therefore $$\hat{f}(s)\Big((1-k)s^{\alpha}-\tfrac{k\alpha}{1-\alpha}\Big)-s^{\alpha-1}f(0)=\tfrac{1-\alpha}{B(\alpha)}\big(s^{\alpha}+\tfrac{\alpha}{1-\alpha}\big)\hat{g}(s).$$
Thus the Laplace-transformed solution $\hat{f}(s)$ is given by
\begin{align*}
\hat{f}(s)&=\Big((1-k)s^{\alpha}-\tfrac{k\alpha}{1-\alpha}\Big)^{-1}\bigg[\tfrac{1-\alpha}{B(\alpha)}\Big(s^{\alpha}+\tfrac{\alpha}{1-\alpha}\Big)\hat{g}(s)+s^{\alpha-1}f(0)\bigg] \\
&=\Big(1-\tfrac{k\alpha}{(1-k)(1-\alpha)}s^{-\alpha}\Big)^{-1}\bigg[\tfrac{1-\alpha}{(1-k)B(\alpha)}\Big(1+\tfrac{\alpha}{1-\alpha}s^{-\alpha}\Big)\hat{g}(s)+\tfrac{1}{(1-k)s}f(0)\bigg] \\
&=\Big(1+\tfrac{k\alpha}{(1-k)(1-\alpha)}s^{-\alpha}+\tfrac{k^2\alpha^2}{(1-k)^2(1-\alpha)^2}s^{-2\alpha}+\dots\Big)\bigg[\tfrac{1-\alpha}{(1-k)B(\alpha)}\Big(1+\tfrac{\alpha}{1-\alpha}s^{-\alpha}\Big)\hat{g}(s)+\tfrac{1}{(1-k)s}f(0)\bigg].
\end{align*}
Now we already know from the previous section that the inverse Laplace transform of the $\hat{g}(s)$ part of the RHS is given by the (equivalent) expressions (\ref{ODE2-soln1}) and (\ref{ODE2-soln2}), except with $h$ replaced by $g$ in each case. And the $f(0)$ part comes to $$\tfrac{1}{1-k}\sum_{n=0}^{\infty}\Big(\tfrac{k\alpha}{(1-k)(1-\alpha)}\Big)^ns^{-n\alpha-1}f(0),$$ so its inverse Laplace transform is
\begin{equation*}
\tfrac{1}{1-k}\sum_{n=0}^{\infty}\Big(\tfrac{k\alpha}{(1-k)(1-\alpha)}\Big)^n\frac{t^{n\alpha}}{\Gamma(n\alpha+1)}f(0)=\tfrac{1}{1-k}E_{\alpha}\Big(\tfrac{k\alpha}{(1-k)(1-\alpha)}t^{\alpha}\Big)f(0).
\end{equation*}
Thus, using equation (\ref{ODE2-soln2}) for the $\hat{g}(s)$ part, we find that the unique Laplace-transformable solution to the ODE (\ref{ODE2}) is given by
\begin{align}
\nonumber f(t)&=-\frac{1}{A}g(t)+\frac{1}{A(1-k)}\frac{\mathrm{d}}{\mathrm{d}t}\int_0^tE_{\alpha}\Big(\tfrac{k\alpha}{(1-k)(1-\alpha)}(t-x)^{\alpha}\Big)g(x)\,\mathrm{d}x+\frac{1}{1-k}E_{\alpha}\Big(\tfrac{k\alpha}{(1-k)(1-\alpha)}t^{\alpha}\Big)f(0) \\
\label{ODE5-soln}
&=\frac{k}{A(1-k)}g(t)+\frac{1}{A(1-k)}\int_0^t\frac{\mathrm{d}}{\mathrm{d}t}E_{\alpha}\Big(\tfrac{k\alpha}{(1-k)(1-\alpha)}(t-x)^{\alpha}\Big)g(x)\,\mathrm{d}x+\frac{1}{1-k}E_{\alpha}\Big(\tfrac{k\alpha}{(1-k)(1-\alpha)}t^{\alpha}\Big)f(0).
\end{align}

By applying the above argument multiple times, we can solve any linear sequential fractional ODE of the form
\begin{equation}
\label{ODE6}
\Big(\prescript{ABC}{}D^{\alpha}_{0+}-A\Big)\Big(\prescript{ABC}{}D^{\beta}_{0+}-B\Big)\dots\Big(\prescript{ABC}{}D^{\gamma}_{0+}-C\Big)f(t)=g(t)
\end{equation}
where $g$ is a Laplace-transformable function and $\alpha,\beta,\dots,\gamma\in(0,1)$ and $A,B,\dots,C$ are constants. In each case, we can construct a nested integral formula for the unique Laplace-transformable solution $f$ to the ODE (\ref{ODE6}), which depends linearly on the initial condition $f(0)$.

\begin{example}
As an example application of this method, let us consider the following sequential fractional ODE, analogous to (\ref{ABR-example}) but with ABC derivatives instead of ABR:
\begin{equation}
\label{ABC-example}
\prescript{ABC}{}D^{1/2}_{0+}\circ\prescript{ABC}{}D^{1/2}_{0+}f(t)=f(t)+g(t),
\end{equation}
where we take the normalisation function $B(\alpha)$ to be identically $1$. As in Example \ref{ODE-example1}, we can split this ODE into a coupled pair as follows:
\begin{align}
\label{ABC-example1}
\prescript{ABC}{}D^{1/2}_{0+}j(t)-j(t)&=g(t); \\
\label{ABC-example2}
\prescript{ABC}{}D^{1/2}_{0+}f(t)+f(t)&=j(t).
\end{align}
Note that both (\ref{ABC-example1}) and (\ref{ABC-example2}) are ODEs in the form of (\ref{ODE5}). In the first case, we have $\alpha=1/2$, $A=1$, and therefore $k=1/2$, so the formula (\ref{ODE5-soln}) becomes \[j(t)=g(t)+2\int_0^t\frac{\mathrm{d}}{\mathrm{d}t}E_{1/2}\big((t-x)^{1/2}\big)g(x)\,\mathrm{d}x+2E_{1/2}\big(t^{1/2}\big)j(0).\] In the second case, we have $\alpha=1/2$, $A=-1$, and therefore $k=-1/2$, so the formula (\ref{ODE5-soln}) becomes \[f(t)=\frac{1}{3}j(t)-\frac{2}{3}\int_0^t\frac{\mathrm{d}}{\mathrm{d}t}E_{1/2}\Big(-\tfrac{1}{3}(t-x)^{1/2}\Big)j(x)\,\mathrm{d}x+\frac{2}{3}E_{1/2}\Big(-\tfrac{1}{3}t^{1/2}\Big)f(0).\] So with initial conditions giving $f(0)=j(0)=0$, for example, the solution to the sequential ODE (\ref{ABR-example}) is:
\begin{multline*}
f(t)=\frac{1}{3}g(t)+\frac{2}{3}\int_0^t\frac{\mathrm{d}}{\mathrm{d}t}E_{1/2}\big((t-x)^{1/2}\big)g(x)\,\mathrm{d}x-\frac{2}{3}\int_0^t\frac{\mathrm{d}}{\mathrm{d}t}E_{1/2}\Big(-\tfrac{1}{3}(t-x)^{1/2}\Big)g(x)\,\mathrm{d}x \\ -\frac{2}{3}\int_0^t\frac{\mathrm{d}}{\mathrm{d}t}E_{1/2}\Big(-\tfrac{1}{3}(t-x)^{1/2}\Big)\Bigg[2\int_0^x\frac{\mathrm{d}}{\mathrm{d}x}E_{1/2}\big((x-y)^{1/2}\big)g(y)\,\mathrm{d}y\Bigg]\,\mathrm{d}x.
\end{multline*}
Thus we see that the only difference between the results for linear ODEs in the ABR model and in the ABC model is in how the initial conditions manifest themselves in the solutions.
\end{example}

Note that various examples of ODEs of the form (\ref{ODE2}) and (\ref{ODE5}) have already found applications in the real world, for example to electrical circuits of type RC, LC, and RL \cite{gomez-aguilar2}. Our analysis goes beyond these to cover general ODEs of the form (\ref{ODE3}) and (\ref{ODE6}), and also some nonlinear ODEs as we are about to see.

\subsection{Nonlinear ODEs by Laplace methods}
Similar methods to those utilised above can also be used to solve certain special classes of nonlinear ordinary differential equations. For example, consider the following ODE, a fractional version of one of the classes of ODE considered in \cite{adomian}:
\begin{equation}
\label{ODE-nonlinear}
\prescript{ABC}{}D^{\alpha}_{0+}f(t)-Af\ast f(t)=g(t),
\end{equation}
where $f$ and $g$ are Laplace-transformable functions, $\alpha\in(0,1)$, $A$ is a constant, and $\ast$ denotes the convolution operation. Taking Laplace transforms of this equation, we get: \[\frac{B(\alpha)}{1-\alpha}\cdot\frac{s^{\alpha-1}}{s^{\alpha}+\tfrac{\alpha}{1-\alpha}}\Big(s\hat{f}(s)-f(0)\Big)-A\big(\hat{f}(s)\big)^2=\hat{g}(s),\] which can be rearranged to \[A\big(\hat{f}(s)\big)^2-\frac{B(\alpha)}{1-\alpha}\cdot\frac{s^{\alpha}}{s^{\alpha}+\tfrac{\alpha}{1-\alpha}}\hat{f}(s)+\Big(\hat{g}(s)+\frac{B(\alpha)}{1-\alpha}\cdot\frac{s^{\alpha-1}}{s^{\alpha}+\tfrac{\alpha}{1-\alpha}}f(0)\Big)=0.\] But this is simply a quadratic equation in $\hat{f}(s)$, so we can solve it to find:
\begin{equation}
\label{ODE-nonlinear-soln}
\hat{f}(s)=\frac{\frac{B(\alpha)}{1-\alpha}\cdot\frac{s^{\alpha}}{s^{\alpha}+\tfrac{\alpha}{1-\alpha}}\pm\sqrt{\bigg(\frac{B(\alpha)}{1-\alpha}\cdot\frac{s^{\alpha}}{s^{\alpha}+\tfrac{\alpha}{1-\alpha}}\bigg)^2-4A\Big(\hat{g}(s)+\frac{B(\alpha)}{1-\alpha}\cdot\frac{s^{\alpha-1}}{s^{\alpha}+\tfrac{\alpha}{1-\alpha}}f(0)\Big)}}{2A}.
\end{equation}
Note that the right hand side of (\ref{ODE-nonlinear-soln}) depends only on $g(t)$ and $f(0)$, so we can take inverse Laplace transforms to get an explicit formula for $f(t)$.

Much the same arguments can of course be applied to the ODE (\ref{ODE-nonlinear}) with the ABC derivative replaced by an ABR one, and the final result would look similar to (\ref{ODE-nonlinear-soln}) except with different dependence on the initial conditions.

As we can see, the final results are less succinct and elegant than those obtained above for linear ODEs, but this is natural since nonlinear ODEs almost always present extra difficulties as compared to the simpler linear case.

\subsection{Nonlinear ODEs using the series formula}
There are also some classes of nonlinear ODEs for which the series formula of Theorem \ref{ABRL-series} enables us to obtain a quick solution. For example, consider the following nonlinear ODE, a case of the Riccati equation, which in the classical Caputo context was considered in \cite{elsaid}:
\begin{equation}
\label{Riccati}
\prescript{ABC}{}D_{0+}^{\alpha}f(t)=P+Q\big(f(t)\big)^2,\quad f(0)=f_0
\end{equation}
We consider the possibility of a convergent series solution of the form
\begin{equation}
\label{Riccati-ansatz}
f(t)=\sum_{k=0}^{\infty}a_kt^{k\alpha},
\end{equation}
whose ABC derivative can be written as follows, using the result of Theorem \ref{ABC-series}.
\begin{align*}
\prescript{ABC}{}D_{0+}^{\alpha}f(t)
&=\frac{B(\alpha)}{1-\alpha}\sum_{n=0}^{\infty}\Big(\frac{-\alpha}{1-\alpha}\Big)^n\prescript{RL}{}I^{n\alpha+1}_{0+}\left(\sum_{k=1}^{\infty}a_kk\alpha t^{k\alpha-1}\right) \\
&=\frac{B(\alpha)}{1-\alpha}\sum_{n=0}^{\infty}\sum_{k=1}^{\infty}\Big(\frac{-\alpha}{1-\alpha}\Big)^n\frac{a_k\Gamma(k\alpha+1)}{\Gamma((k+n)\alpha+1)}t^{(k+n)\alpha} \\
&=\sum_{m=1}^{\infty}\Bigg[\frac{B(\alpha)}{1-\alpha}\sum_{k=1}^{m}a_k\big(\tfrac{-\alpha}{1-\alpha}\big)^{m-k}\frac{\Gamma(k\alpha+1)}{\Gamma(m\alpha+1)}\Bigg]t^{m\alpha}.
\end{align*}
Meanwhile, the right hand side of (\ref{Riccati}), with the ansatz (\ref{Riccati-ansatz}), is simply \[P+Q\big(f(t)\big)^2=P+Q\sum_{k=0}^{\infty}a_kt^{k\alpha}\sum_{n=0}^{\infty}a_nt^{n\alpha}
=\sum_{m=0}^{\infty}\Bigg[P\delta_{m0}+Q\sum_{k=0}^ma_ka_{m-k}\Bigg]t^{m\alpha}.\] Thus the nonlinear Riccati equation (\ref{Riccati}) has a solution of the form (\ref{Riccati-ansatz}) with the coefficients $a_k$ defined to satisfy
\begin{align*}
0&=P+Qa_0^2; \\
\tfrac{B(\alpha)}{1-\alpha}\sum_{k=1}^{m}a_k\big(\tfrac{-\alpha}{1-\alpha}\big)^{m-k}\frac{\Gamma(k\alpha+1)}{\Gamma(m\alpha+1)}&=Q\sum_{k=0}^ma_ka_{m-k}\quad\text{ for all $m>0$.}
\end{align*}
This identity can be used to find all the coefficients, by solving it for each value of $m$ in turn. For $m=0$, we have $0=P+Qa_0^2$, 
so $a_0=\pm\sqrt{-P/Q}$. For $m>0$, we have \[a_m=\frac{\frac{B(\alpha)}{1-\alpha}\sum_{k=1}^{m-1}a_k\left(\frac{-\alpha}{1-\alpha}\right)^{m-k}\frac{\Gamma(k\alpha+1)}{\Gamma(m\alpha+1)}-Q\sum_{k=1}^{m-1}a_ka_{m-k}}{2Qa_0-\frac{B(\alpha)}{1-\alpha}}.\] Thus, the general solution of the form (\ref{Riccati-ansatz}) to the Riccati equation (\ref{Riccati}) is
\begin{equation*}
f(t)=\sqrt{\frac{-P}{Q}}+\sum_{m=1}^{\infty}\frac{\frac{B(\alpha)}{1-\alpha}\sum_{k=1}^{m-1}a_k\left(\frac{-\alpha}{1-\alpha}\right)^{m-k}\frac{\Gamma(k\alpha+1)}{\Gamma(m\alpha+1)}-Q\sum_{k=1}^{m-1}a_ka_{m-k}}{2Qa_0-\frac{B(\alpha)}{1-\alpha}}t^{m\alpha}.
\end{equation*}

This is an extension to the AB model of the results of \cite{elsaid} for series solutions of the Riccati equation. Note that this simple solution was only possible because we were able to use the new series formula (\ref{ABC-series-formula}) for AB derivatives. The fractional derivative on the left hand side of the ODE became a double series to match the double series on the right hand side, and we could solve the resulting identity very directly.

\section{The semigroup property}
\label{sec-semigroup}
The semigroup property for AB fractional differintegrals is \textit{not} satisfied in general. For example, taking $B(\alpha)=1$ we get $$\prescript{AB}{}I^{2/3}_{0+}(t)=\Big(\tfrac{1}{3}+\tfrac{2}{3}\prescript{RL}{}I^{2/3}_{0+}\Big)t=\tfrac{1}{3}t+\tfrac{2}{3\Gamma(8/3)}t^{5/3}$$ and yet
\begin{align*}
\prescript{AB}{}I^{1/3}_{0+}\Big(\prescript{AB}{}I^{1/3}_{0+}(t)\Big)&=\prescript{AB}{}I^{1/3}_{0+}\Big(\tfrac{2}{3}t+\tfrac{1}{3\Gamma(7/3)}t^{4/3}\Big) \\
&=\tfrac{2}{3}\Big(\tfrac{2}{3}t+\tfrac{1}{3\Gamma(7/3)}t^{4/3}\Big)+\tfrac{1}{3\Gamma(7/3)}\Big(\tfrac{2}{3}t^{4/3}+\tfrac{\Gamma(7/3)}{\Gamma(8/3)}t^{5/3}\Big) \\
&=\tfrac{4}{9}t+\tfrac{4}{9\Gamma(7/3)}t^{4/3}+\tfrac{1}{3\Gamma(8/3)}t^{5/3}
\end{align*}
-- two entirely different expressions.

Can we find conditions for when the semigroup property \textit{does} hold?

Firstly, note that it will be sufficient to consider fractional \textit{integrals} only. Any function which satisfies the semigroup property for ABR fractional derivatives generates one which satisfies it for AB fractional integrals, and vice versa. This is because $$\prescript{ABR}{}I^{\alpha}_{0+}\Big(\prescript{ABR}{}I^{\beta}_{0+}f(t)\Big)=g(t)=\prescript{ABR}{}I^{\alpha+\beta}_{0+}f(t)$$ is exactly equivalent to $$\prescript{ABR}{}D^{\beta}_{0+}\Big(\prescript{ABR}{}D^{\alpha}_{0+}g(t)\Big)=f(t)=\prescript{ABR}{}D^{\alpha+\beta}_{0+}g(t).$$ This is good to know, because the definition of AB fractional integrals is much simpler and easier to work with than that of ABR fractional derivatives.

The semigroup property for AB fractional integrals is equivalent to the following conditions (where we use $\prescript{AB}{}I$ to denote AB fractional integrals and just $I$ to denote Riemann--Liouville fractional integrals):
\begin{align}
\nonumber &\,\prescript{AB}{}I^{\alpha}_{0+}\Big(\prescript{AB}{}I^{\beta}_{0+}f(t)\Big)=\prescript{AB}{}I^{\alpha+\beta}_{0+}f(t) \\
\nonumber \Leftrightarrow&\,\bigg(\frac{1-\alpha}{B(\alpha)}+\frac{\alpha}{B(\alpha)}I^{\alpha}\bigg)\bigg(\frac{1-\beta}{B(\beta)}+\frac{\beta}{B(\beta)}I^{\beta}\bigg)f(t)=\bigg(\frac{1-\alpha-\beta}{B(\alpha+\beta)}+\frac{\alpha+\beta}{B(\alpha+\beta)}I^{\alpha+\beta}\bigg)f(t) \\
\nonumber
\begin{split}
\Leftrightarrow&\,\frac{(1-\alpha)(1-\beta)}{B(\alpha)B(\beta)}f+\frac{\alpha(1-\beta)}{B(\alpha)B(\beta)}I^{\alpha}f+\frac{\beta(1-\alpha)}{B(\alpha)B(\beta)}I^{\beta}f+\frac{\alpha\beta}{B(\alpha)B(\beta)}I^{\alpha+\beta}f \\
&\hspace{8cm}=\frac{1-\alpha-\beta}{B(\alpha+\beta)}f+\frac{\alpha+\beta}{B(\alpha+\beta)}I^{\alpha+\beta}f
\end{split} \\
\label{semigroup-FIE}
\begin{split}
\Leftrightarrow&\,\bigg[\frac{\alpha\beta}{B(\alpha)B(\beta)}-\frac{\alpha+\beta}{B(\alpha+\beta)}\bigg]I^{\alpha+\beta}f+\frac{\alpha(1-\beta)}{B(\alpha)B(\beta)}I^{\alpha}f+\frac{\beta(1-\alpha)}{B(\alpha)B(\beta)}I^{\beta}f \\
&\hspace{7cm}+\bigg[\frac{(1-\alpha)(1-\beta)}{B(\alpha)B(\beta)}-\frac{1-\alpha-\beta}{B(\alpha+\beta)}\bigg]f=0.
\end{split}
\end{align}
Thus we have a Riemann--Liouville fractional integral equation in $f$ which must be satisfied in order for the AB fractional integrals of $f$ to have the semigroup property. If we assume for simplicity that the normalisation function $B$ satisfies its own semigroup property $B(\alpha)B(\beta)=B(\alpha+\beta)$, then we can simplify the condition (\ref{semigroup-FIE}) as follows:
\begin{align*}
(\ref{semigroup-FIE})&\Leftrightarrow\big(\alpha\beta-\alpha-\beta\big)I^{\alpha+\beta}f +\big(\alpha-\alpha\beta\big)I^{\alpha}f+\big(\beta-\beta\alpha\big)I^{\beta}f +\alpha\beta f=0 \\
&\Leftrightarrow\alpha\beta\Big(I^{\alpha+\beta}f-I^{\alpha}f-I^{\beta}f+f\Big)+\alpha\big(I^{\alpha}f-I^{\alpha+\beta}f\Big)+\beta\big(I^{\beta}f-I^{\alpha+\beta}f\big)=0 \\
&\Leftrightarrow\alpha\beta\big(I^{\alpha}-1\big)\big(I^{\beta}-1\big)f-\alpha I^{\alpha}\big(I^{\beta}-1\big)f-\beta I^{\beta}\big(I^{\alpha}-1\big)f=0 \\
&\Leftrightarrow\Big[\alpha\big(I^{\beta}-1\big)-I^{\beta}\Big]\Big[\beta\big(I^{\alpha}-1\big)-I^{\alpha}\Big]f-I^{\alpha+\beta}f=0 \\
&\Leftrightarrow\Big[(\alpha-1)I^{\beta}-\alpha\Big]\Big[(\beta-1)I^{\alpha}-\beta\Big]f =I^{\alpha+\beta}f.
\end{align*}

Using the composition properties of Riemann--Liouville fractional differintegrals, we can derive a necessary condition for the semigroup property in the form of a Riemann--Liouville fractional \textit{differential} equation by applying $D^{\alpha+\beta}=\prescript{RL}{}D^{\alpha+\beta}_{0+}$ to (\ref{semigroup-FIE}):
\begin{multline}
\label{semigroup-FDE}
(\ref{semigroup-FIE})\Rightarrow\bigg[\frac{(1-\alpha)(1-\beta)}{B(\alpha)B(\beta)}-\frac{1-\alpha-\beta}{B(\alpha+\beta)}\bigg]D^{\alpha+\beta}f+\frac{\alpha(1-\beta)}{B(\alpha)B(\beta)}D^{\beta}f+\frac{\beta(1-\alpha)}{B(\alpha)B(\beta)}D^{\alpha}f \\ +\bigg[\frac{\alpha\beta}{B(\alpha)B(\beta)}-\frac{\alpha+\beta}{B(\alpha+\beta)}\bigg]f=0
\end{multline}
Assuming that $B(\alpha)B(\beta)=B(\alpha+\beta)$, we can again rewrite this condition in a more elegant form:
\begin{equation*}
(\ref{semigroup-FDE})\Leftrightarrow\alpha\beta D^{\alpha+\beta}f+\alpha(1-\beta)D^{\beta}f+\beta(1-\alpha)D^{\alpha}f+ (\alpha\beta-\alpha-\beta)f=0.
\end{equation*}
Using the methods described in \cite{miller}, we can solve this fractional ODE for rational $\alpha,\beta$ by finding the roots of the indicial polynomial
\begin{equation*}
P(x)=\alpha\beta x^{\alpha+\beta}+\alpha(1-\beta)x^{\beta}+\beta(1-\alpha)x^{\alpha}+ (\alpha\beta-\alpha-\beta)=(\beta x^{\alpha}-\beta+1)(\alpha x^{\beta}-\alpha+1)-1.
\end{equation*}
This equation is not going to have neat solutions in general, but in the case where $\alpha=\beta$ things become easier. In this case, our assumption on $B$ from earlier becomes $B(\alpha)^2=B(2\alpha)$ and the indicial polynomial is
\begin{align*}
P(x)&=\alpha^2x^{2\alpha}+2\alpha(1-\alpha)x^{\alpha}+ (\alpha^2-2\alpha)=\alpha^2\big(x^{2\alpha}+\tfrac{2-2\alpha}{\alpha}x^{\alpha}+ \tfrac{\alpha-2}{\alpha}\big) \\
&=\alpha^2\big(x-1\big)\big(x-\tfrac{\alpha-2}{\alpha}\big).
\end{align*}
Here the indicial polynomial is relatively easy to solve, and we can use its roots to construct a solution $f$ to the fractional ODE (\ref{semigroup-FDE}) in the form of a linear combination of incomplete gamma functions.

\begin{example}
For example, let us consider the simplest case, that in which $\alpha=\beta=\tfrac{1}{q}$ for some natural number $q>2$. Here we have $$P'(x)=2\alpha(\alpha x-\alpha+1), P'(1)=2\alpha,P'\big(\tfrac{\alpha-2}{\alpha}\big)=-2\alpha$$ and by \cite{miller} the solution to the fractional ODE (\ref{semigroup-FDE}) is
\begin{equation}
\label{semigroup-soln}
f(t)=\frac{1}{2\alpha}\sum_{k=0}^{q-1}E_t(-k\alpha,1)-\frac{1}{2\alpha}\sum_{k=0}^{q-1}\big(\tfrac{\alpha-2}{\alpha}\big)^{q-k-1}E_t\big(-k\alpha,\big(\tfrac{\alpha-2}{\alpha}\big)^q\big)
\end{equation}
where $E$ in this case is the relative of the Mittag-Leffler function defined by $E_t(\nu,a)=z^{\nu}\sum_{n=0}^{\infty}\tfrac{(at)^n}{\Gamma(\nu+n+1)}$.

So the function $f$ defined by (\ref{semigroup-soln}) is the only function satisfying the semigroup property $$\prescript{ABR}{}I^{\alpha}_{0+}\Big(\prescript{ABR}{}I^{\alpha}_{0+}f(t)\Big)=\prescript{ABR}{}I^{2\alpha}_{0+}f(t)$$ where $\alpha=\tfrac{1}{q}$, $q>2$ is a natural number, and $B$ satisfies $B(\alpha)^2=B(2\alpha)$.
\end{example}

\section{The product rule}
\label{sec-product}
Among other results \cite{osler3,osler4,osler5} on RL differintegrals, Osler proved \cite{osler1} that the product rule can be generalised to the following identity for Riemann--Liouville fractional differintegrals:
\begin{equation}
\label{Osler-Leibniz}
\prescript{RL}{}D^{\alpha}_{a+}\big(u(t)v(t)\big)=\sum_{n=0}^{\infty}\binom{\alpha}{n}\prescript{RL}{}D^{\alpha-n}_{a+}u(t)\prescript{RL}{}D^{n}_{a+}v(t)\hspace{1cm}\forall\alpha\in\mathbb{C},t\in R\backslash\{a\}
\end{equation}
where $u(t)$, $v(t)$, and $u(t)v(t)$ are all functions in the form $t^{\lambda}\eta(t)$ with $\mathrm{Re}(\lambda)>-1$ and $\eta$ analytic on a domain $R\subset\mathbb{C}$ containing $a$. Note that $\mathrm{Re}(\alpha)$ can be either positive or negative, so this identity covers fractional integrals as well as derivatives. Thus for $\alpha>0$ we have
\begin{equation}
\prescript{RL}{}I^{\alpha}_{a+}\big(u(t)v(t)\big)=\sum_{n=0}^{\infty}\binom{-\alpha}{n}\prescript{RL}{}I^{\alpha+n}_{a+}u(t)\frac{\mathrm{d}^nv}{\mathrm{d}t^n},
\end{equation}
which we can use to prove a corresponding identity for ABR fractional derivatives. For any $\alpha\in(0,1)$ and $u,v$ as above,
\begin{align}
\nonumber \prescript{ABR}{}D^{\alpha}_{a+}\big(u(t)v(t)\big)&=\frac{B(\alpha)}{1-\alpha}\sum_{n=0}^{\infty}\Big(\frac{-\alpha}{1-\alpha}\Big)^n\prescript{RL}{}I^{\alpha n}_{a+}\big(u(t)v(t)\big) \\
\nonumber &=\frac{B(\alpha)}{1-\alpha}\sum_{n=0}^{\infty}\sum_{m=0}^{\infty}\Big(\frac{-\alpha}{1-\alpha}\Big)^n\binom{-n\alpha}{m}\prescript{RL}{}I^{\alpha n+m}_{a+}u(t)\frac{\mathrm{d}^mv}{\mathrm{d}t^m} \\
\label{AB-Leibniz}
&=\sum_{m=0}^{\infty}\frac{\mathrm{d}^mv}{\mathrm{d}t^m}\Bigg[\frac{B(\alpha)}{1-\alpha}\sum_{n=0}^{\infty}\Big(\frac{-\alpha}{1-\alpha}\Big)^n\binom{-n\alpha}{m}\prescript{RL}{}I^{\alpha n+m}_{a+}u(t)\Bigg].
\end{align}
Note that when $m=0$, the term in square brackets is exactly $\prescript{ABR}{}D^{\alpha}_{a+}u(t)$.

In order to prove rigorously that the double series in (\ref{AB-Leibniz}) converges, we shall express it as the sum of a finite series and a remainder term, and then prove that the remainder term tends to zero. Specifically, we start from the following formula, equation (2.199) in \cite{podlubny}:
\begin{equation}
\label{Leibniz-remainder}
\prescript{RL}{}D^{\alpha}_{a+}\big(u(t)v(t)\big)=\sum_{n=0}^{N}\binom{\alpha}{n}\prescript{RL}{}D^{\alpha-n}_{a+}u(t)\prescript{RL}{}D^{n}_{a+}v(t)-R_N^{\alpha}(t),
\end{equation}
valid for $\alpha\in\mathbb{R},n\geq\alpha+1,u\in C[a,t],v\in C^{N+1}[a,t]$, where the remainder term $R_N^{\alpha}(t)$ is defined by $$R_N^{\alpha}(t)=\frac{1}{N!\Gamma(-\alpha)}\int_a^t(t-\tau)^{-\alpha-1}u(\tau)\bigg[\int_{\tau}^tv^{(N+1)}(\xi)(\tau-\xi)^N\,\mathrm{d}\xi\bigg]\mathrm{d}\tau.$$

Applying (\ref{Leibniz-remainder}) with $\alpha$ replaced by $-n\alpha$ for $n\in\mathbb{N}$, we find
\begin{align}
\nonumber \prescript{ABR}{}D^{\alpha}_{a+}\big(u(t)v(t)\big)&=\frac{B(\alpha)}{1-\alpha}\sum_{n=0}^{\infty}\Big(\frac{-\alpha}{1-\alpha}\Big)^n\prescript{RL}{}I^{\alpha n}_{a+}\big(u(t)v(t)\big) \\
\nonumber &=\frac{B(\alpha)}{1-\alpha}\sum_{n=0}^{\infty}\Bigg[\sum_{m=0}^{N}\Big(\frac{-\alpha}{1-\alpha}\Big)^n\binom{-n\alpha}{m}\prescript{RL}{}D^{-n\alpha-m}_{a+}u(t)\prescript{RL}{}D^{m}_{a+}v(t)-R_N^{-n\alpha}(t)\Bigg] \\
\label{AB-Leibniz-remainder}
&=\sum_{m=0}^{N}\frac{\mathrm{d}^mv}{\mathrm{d}t^m}\Bigg[\frac{B(\alpha)}{1-\alpha}\sum_{n=0}^{\infty}\Big(\frac{-\alpha}{1-\alpha}\Big)^n\binom{-n\alpha}{m}\prescript{RL}{}I^{\alpha n+m}_{a+}u(t)\Bigg]-\frac{B(\alpha)}{1-\alpha}\sum_{n=0}^{\infty}R_N^{-n\alpha}(t),
\end{align}
where the interchange of summations is valid because by Theorem \ref{ABRL-series} we know that the sum over $n$ converges locally uniformly whenever the relevant ABR derivative is well-defined. So in order to establish the convergence of the outer series in (\ref{AB-Leibniz}), it will suffice to prove that $$\lim_{N\rightarrow\infty}\sum_{n=0}^{\infty}R_N^{-n\alpha}(t)=0$$ for $\phi\in C^{\infty}[a,t]$. And this can be proven by almost exactly the same argument as in \cite{podlubny}. Equation (2.201) from there gives $$R_N^{-n\alpha}(t)=\frac{(-1)^N(t-a)^{N+n\alpha+1}}{N!\Gamma(n\alpha)}\int_0^1\int_0^1u\big(a+\eta(t-a)\big)v^{(N+1)}\big(a+(t-a)(\zeta+\eta-\zeta\eta)\big)\,\mathrm{d}\eta\,\mathrm{d}\zeta,$$ so
\begin{multline*}
\sum_{n=0}^{\infty}R_N^{-n\alpha}(t)=\frac{(-1)^{N}(t-a)^{N+2}}{N!}\frac{\mathrm{d}}{\mathrm{d}t}\Big(E_{\alpha}\big((t-a)^{\alpha}\big)\Big)\times \\ \int_0^1\int_0^1u\big(a+\eta(t-a)\big)v^{(N+1)}\big(a+(t-a)(\zeta+\eta-\zeta\eta)\big)\,\mathrm{d}\eta\,\mathrm{d}\zeta,
\end{multline*}
which converges to $0$ as $N\rightarrow\infty$, as required. Note that in order to derive the final expression, we used the fact that $$\sum_{n=0}^{\infty}\frac{t^{-n\alpha}}{\Gamma(n\alpha)}=-\sum_{n=0}^{\infty}\frac{-n\alpha t^{-n\alpha}}{\Gamma(n\alpha+1)}=-t\frac{\mathrm{d}}{\mathrm{d}t}\Big(\sum_{n=0}^{\infty}\frac{t^{-n\alpha}}{\Gamma(n\alpha+1)}\Big)=-t\frac{\mathrm{d}}{\mathrm{d}t}E_{\alpha}\big(t^{-\alpha}\big).$$ Thus (\ref{AB-Leibniz}) is valid as a generalisation of the Leibniz rule to ABR fractional derivatives.

\begin{example}
As an example to verify this new result, let us take $u(t)=t^2$ and $v(t)=t$ with $a=0$. Then the right-hand side of (\ref{AB-Leibniz}) becomes
\begin{align*}
&\sum_{m=0}^{1}\frac{\mathrm{d}^m}{\mathrm{d}t^m}(t)\Bigg[\frac{B(\alpha)}{1-\alpha}\sum_{n=0}^{\infty}\Big(\frac{-\alpha}{1-\alpha}\Big)^n\binom{-n\alpha}{m}\prescript{RL}{}I^{\alpha n+m}_{0+}(t^2)\Bigg] \\
=\;&t\Bigg[\frac{B(\alpha)}{1-\alpha}\sum_{n=0}^{\infty}\Big(\frac{-\alpha}{1-\alpha}\Big)^n\frac{2}{\Gamma(3+\alpha n)}t^{2+\alpha n}\Bigg]+\Bigg[\frac{B(\alpha)}{1-\alpha}\sum_{n=0}^{\infty}\Big(\frac{-\alpha}{1-\alpha}\Big)^n(-n\alpha)\frac{2}{\Gamma(4+\alpha n)}t^{3+\alpha n}\Bigg] \\
=\;&\frac{B(\alpha)}{1-\alpha}\sum_{n=0}^{\infty}\Big(\frac{-\alpha}{1-\alpha}\Big)^n\bigg[\frac{2}{\Gamma(3+\alpha n)}-\frac{2n\alpha}{\Gamma(4+\alpha n)}\bigg]t^{3+\alpha n} \\
=\;&\frac{B(\alpha)}{1-\alpha}\sum_{n=0}^{\infty}\Big(\frac{-\alpha}{1-\alpha}\Big)^n\frac{6}{\Gamma(4+\alpha n)}t^{3+\alpha n}=\frac{B(\alpha)}{1-\alpha}\sum_{n=0}^{\infty}\Big(\frac{-\alpha}{1-\alpha}\Big)^n\prescript{RL}{}I^{\alpha n}_{0+}(t^3)=\prescript{ABR}{}D^{\alpha}_{0+}(t^3),
\end{align*}
exactly as expected. \qed
\end{example}

The product rule is an extremely important result to examine in any new model of fractional calculus. As was pointed out by Tarasov \cite{tarasov}, the Leibniz rule plays a crucial role in fractional calculus and its applications, to the extent that it can be used as a test for the validity of a given model.

Furthermore, having an analogue of the product rule enables us to greatly extend the number of functions whose fractional derivatives can be easily calculated. We now have a formula which can be used to compute AB derivatives of anything which can be expressed as a product of two or more functions whose AB derivatives are already known -- thus expanding the space of functions on which we can easily do calculations. The expression (\ref{AB-Leibniz}) is admittedly cumbersome, but this is a common problem with fractional generalisations of results from calculus (we see it also in the classical result (\ref{Osler-Leibniz})), and it is still computationally manageable.

\section{The chain rule}
\label{sec-chain}
\subsection{Deriving the chain rule}
Osler also generalised the chain rule to the following identity for fractional Riemann--Liouville differintegrals of composite functions \cite{osler2}:
\begin{equation}
\label{Osler-chain}
\prescript{RL}{}D^{\alpha}_{a+}f(g(t))=\frac{(t-a)^{-\alpha}}{\Gamma(1-\alpha)}f(g(t))+\sum_{n=1}^{\infty}\binom{\alpha}{n}\frac{(t-a)^{n-\alpha}}{\Gamma(n-\alpha+1)}\sum_{k=1}^n\frac{\mathrm{d}^kf(g(t))}{\mathrm{d}g(t)^k}\sum_{(P_1,\dots,P_n)}\Bigg[\prod_{i=1}^n\tfrac{i}{P_i!(i!)^{P_i}}\Big(\tfrac{\mathrm{d}^ig(t)}{\mathrm{d}t^i}\Big)^{P_i}\Bigg]
\end{equation}
for all $t,\alpha,a\in\mathbb{C}$, where $g$ is smooth and $f(g(t))$ is a function of the form $t^{\lambda}\eta(t)$ with $\mathrm{Re}(\lambda)>-1$ and $\eta$ analytic on a domain $R\subset\mathbb{C}$ containing $a$, and the final sum is over all $n$-tuples $(P_1,\dots,P_n)$ of non-negative integers such that $\sum_iP_i=k$ and $\sum_iiP_i=n$. Note that again this identity covers fractional integrals as well as derivatives, so we can use it to get a corresponding identity for ABR fractional derivatives. For $\alpha\in(0,1)$ and any $f,g$ as above such that $f(g(t))$ is an $L^1$ function,
\begin{align}
\nonumber \prescript{ABR}{}D^{\alpha}_{a+}f(g(t))&=\frac{B(\alpha)}{1-\alpha}\sum_{m=0}^{\infty}\Big(\frac{-\alpha}{1-\alpha}\Big)^m\prescript{RL}{}I^{\alpha m}_{a+}f(g(t)) \\
\begin{split}
\nonumber &=\frac{B(\alpha)}{1-\alpha}\sum_{m=0}^{\infty}\Big(\frac{-\alpha}{1-\alpha}\Big)^m\Bigg[\frac{(t-a)^{m\alpha}}{\Gamma(m\alpha+1)}f(g(t)) \\
\nonumber &\hspace{1cm}+\sum_{n=1}^{\infty}\binom{-m\alpha}{n}\frac{(t-a)^{n+m\alpha}}{\Gamma(n+m\alpha+1)}\sum_{k=1}^n\frac{\mathrm{d}^kf(g(t))}{\mathrm{d}g(t)^k}\sum_{(P_1,\dots,P_n)}\Bigg[\prod_{i=1}^n\tfrac{i}{P_i!(i!)^{P_i}}\Big(\tfrac{\mathrm{d}^ig(t)}{\mathrm{d}t^i}\Big)^{P_i}\Bigg]\Bigg]
\end{split} \\
\label{AB-chain}
&=\frac{B(\alpha)}{1-\alpha}\bigg[E_{\alpha}\Big(\tfrac{-\alpha}{1-\alpha}(t-a)^{\alpha}\Big)f(g(t))+\sum_{m=0}^{\infty}\sum_{n=1}^{\infty}\big(\tfrac{-\alpha}{1-\alpha}\big)^m\tbinom{-m\alpha}{n}\frac{(t-a)^{n+m\alpha}}{\Gamma(n+m\alpha+1)}\sum_{k=1}^n(\dots)\bigg],
\end{align}
where the series in (\ref{Osler-chain}) are convergent by the proof in \cite{osler2} and the outer series in (\ref{AB-chain}) is locally uniformly convergent by Theorem \ref{ABRL-series}, therefore the sums can be interchanged and the result is a well-defined convergent series.

\begin{example}
As an example to verify this new identity, let us take $a=0$ and $g(t)=e^t$ and $f(t)=t^2$ so that $f(g(t))=e^{2t}$. Then the $k$th derivative of $f(g(t))$ with respect to $g(t)$ is $2e^t$ if $k=1$, $2$ if $k=2$, and $0$ if $k>2$. For $k=1$, we must have $(P_1,\dots,P_n)=(0,\dots,0,1)$ and therefore $$\prod_{i=1}^n\tfrac{i}{P_i!(i!)^{P_i}}\Big(\tfrac{\mathrm{d}^ig(t)}{\mathrm{d}t^i}\Big)^{P_i}=\frac{\prod_{i=1}^ni}{1!(n!)^1}\big(e^t\big)^1=e^t.$$ For $k=2$, we must have either $P_j=P_{n-j}=1$ for some $j\neq\tfrac{n}{2}$ and all other $P_i=0$ or (if $n$ is even) $P_{n/2}=2$ and all other $P_i=0$. In the first case, $$\prod_{i=1}^n\tfrac{i}{P_i!(i!)^{P_i}}\Big(\tfrac{\mathrm{d}^ig(t)}{\mathrm{d}t^i}\Big)^{P_i}=\frac{\prod_{i=1}^ni}{(1!)^2j!(n-j)!}\big(e^t\big)^1\big(e^t\big)^1=\binom{n}{j}e^{2t}$$ while in the second case, $$\prod_{i=1}^n\tfrac{i}{P_i!(i!)^{P_i}}\Big(\tfrac{\mathrm{d}^ig(t)}{\mathrm{d}t^i}\Big)^{P_i}=\frac{\prod_{i=1}^ni}{2!(\tfrac{n}{2}!)^2}\big(e^t\big)^2=\frac{1}{2}\binom{n}{n/2}e^{2t}.$$ So the right-hand side of (\ref{AB-chain}) becomes
\begin{equation*}
\frac{B(\alpha)}{1-\alpha}\bigg[E_{\alpha}\Big(\tfrac{-\alpha}{1-\alpha}t^{\alpha}\Big)e^{2t}+\sum_{m=0}^{\infty}\sum_{n=1}^{\infty}\big(\tfrac{-\alpha}{1-\alpha}\big)^m\tbinom{-m\alpha}{n}\frac{t^{n+m\alpha}}{\Gamma(n+m\alpha+1)}\bigg((2e^{t})e^t+(2)\sum_{j=1}^{\lfloor n/2\rfloor}\binom{n}{j}e^{2t}+\binom{n}{n/2}e^{2t}\bigg)\bigg],
\end{equation*}
where the last term is present only if $n$ is even. This simplifies to
\begin{align*}
&=\frac{B(\alpha)}{1-\alpha}\bigg[E_{\alpha}\Big(\tfrac{-\alpha}{1-\alpha}t^{\alpha}\Big)e^{2t}+\sum_{m=0}^{\infty}\sum_{n=1}^{\infty}\big(\tfrac{-\alpha}{1-\alpha}\big)^m\tbinom{-m\alpha}{n}\frac{t^{n+m\alpha}}{\Gamma(n+m\alpha+1)}\sum_{j=0}^n\tbinom{n}{j}e^{2t}\bigg] \\
&=\frac{B(\alpha)}{1-\alpha}\bigg[\sum_{m=0}^{\infty}\big(\tfrac{-\alpha}{1-\alpha}\big)^m\frac{t^{\alpha m}}{\Gamma(\alpha m+1)}e^{2t}+\sum_{m=0}^{\infty}\sum_{n=1}^{\infty}\big(\tfrac{-\alpha}{1-\alpha}\big)^m\tbinom{-m\alpha}{n}\frac{t^{n+m\alpha}}{\Gamma(n+m\alpha+1)}2^ne^{2t}\bigg] \\
&=\frac{B(\alpha)}{1-\alpha}\sum_{m=0}^{\infty}\sum_{n=0}^{\infty}\big(\tfrac{-\alpha}{1-\alpha}\big)^m\tbinom{-m\alpha}{n}\frac{t^{n+m\alpha}}{\Gamma(n+m\alpha+1)}2^ne^{2t},
\end{align*}
which is exactly the formula for $\prescript{ABR}{}D^{\alpha}_{0+}(e^{2t})$, since the Riemann--Liouville integrals of the exponential function are given by $\prescript{RL}{}I^{m\alpha}_{0+}(e^{2t})=\sum_{n=0}^{\infty}\frac{2^nt^{m\alpha+n}}{\Gamma(m\alpha+n+1)}$. \qed
\end{example}

Once again, being able to use the chain rule enables a big extension to the class of functions with easily computable fractional derivatives. With the results of this and the previous section, we can now explicitly compute AB derivatives of anything which can be derived by multiplication and composition from two or more functions with already known AB derivatives. So once more we have expanded the space of functions on which it is easy to perform calculations with AB derivatives. The expression (\ref{AB-chain}) is even more cumbersome than (\ref{AB-Leibniz}), but as before it is still feasible to compute.

\subsection{An application to fractional mechanics}
In the following section we present a physical application of our results and series representation of AB derivatives. The idea is to show how the expression of the chain rule proved above plays a vital role in writing down the Euler-Lagrange equations of two Lagrangians which differ by a total time-like derivative. From the classical viewpoint these Lagrangians are equivalent, but when we fractionalise the classical Lagrangians this property is lost.

In \cite{baleanu2} the authors used the series expression of the chain rule for Riemann-Liouville fractional derivatives \cite{podlubny} in order to propose a Lagrangian and Hamiltonian representation as a 1+1 field theory. Their starting point was the following expression:
\begin{equation}
\label{Lagrangian}
L'_f=L_f\big(q^{\rho}(t),D_{a+}^{\alpha}q^{\rho}(t)\big)+D_{a+}^{\alpha}F\big(q^m(t)\big),\quad\rho=1,\dots,n
\end{equation}
In our case, the representation of the ABR and ABC derivatives as series, together with the form (\ref{AB-chain}) of the chain rule from the previous subsection, lead us to a generalized version of the results reported in \cite{baleanu2} which are valid now on a bigger domain. Our analogue of (\ref{Lagrangian}) is
\begin{equation}
\label{Lagrangian-AB}
\mathcal{L}'_f=\mathcal{L}_f\big(p^{\gamma}(t),\prescript{ABR}{}D_{a+}^{\alpha}p^{\gamma}(t)\big)+\prescript{ABR}{}D_{a+}^{\alpha}G\big(p^m(t)\big),\quad\gamma=1,\dots,n,
\end{equation}
and again we need the chain rule -- in this case, equation (\ref{AB-chain}) -- in order to evaluate the final term:
\begin{multline*}
\mathcal{L}'_f=\mathcal{L}_f\big(p^{\gamma}(t),\prescript{ABR}{}D_{a+}^{\alpha}p^{\gamma}(t)\big)+\frac{B(\alpha)}{1-\alpha}\Bigg[E_{\alpha}\Big(\tfrac{-\alpha}{1-\alpha}(t-a)^{\alpha}\Big)G(p^m(t))+ \\ \sum_{m=0}^{\infty}\sum_{n=1}^{\infty}\big(\tfrac{-\alpha}{1-\alpha}\big)^m\tbinom{-m\alpha}{n}\frac{(t-a)^{n+m\alpha}}{\Gamma(n+m\alpha+1)}\sum_{k=1}^n\frac{\mathrm{d}^kG(p^m(t))}{\mathrm{d}p^m(t)^k}\sum_{(P_1,\dots,P_n)}\Bigg[\prod_{i=1}^n\tfrac{i}{P_i!(i!)^{P_i}}\Big(\tfrac{\mathrm{d}^ig(t)}{\mathrm{d}t^i}\Big)^{P_i}\Bigg]\Bigg].
\end{multline*}
Note that this expression involves infinitely many derivatives of $p(t)$, giving a non-local theory which will give rise to interesting new Euler-Lagrange equations and Hamiltonian constructions. More precisely, we can say that the dynamical variable $p(t)$ is given by a $1+1$ dimensional field $Q(x,t)$ fulfilling the following chirality condition \cite{baleanu2, bering}: \[\frac{\mathrm{d}Q(x,t)}{\mathrm{d}t}=\partial_xQ(x,t).\]

Thus, the results reported in our paper can lead to a new fractional mechanics.

\section{Conclusions}
\label{sec-conclusions}
In this paper, we have considered various properties of differintegration and fractional differintegration and how they extend to the case where our fractional derivatives are defined using a Mittag-Leffler kernel. These results can be seen as a continuation of the work of \cite{atangana} -- in which, after the definition of fractional derivatives with Mittag-Leffler kernel was established, results on Laplace transforms and Lipschitz-type bounds for such derivatives were proved -- and a parallel development alongside e.g. \cite{abdeljawad2} and \cite{djida}, working on building the fundamental theory of the AB model of fractional calculus and establishing results which can be used later in more advanced study and applications.

First of all, in section \ref{sec-series}, we established a new formula for these derivatives, both those of ABR type and those of ABC type. This formula is in the form of an infinite series of Riemann--Liouville fractional integrals, and is in some contexts easier to handle than the original formula established in \cite{atangana}, since it no longer necessitates dealing with the transcendental Mittag-Leffler function. In our opinion, the new series formula gives a deeper connection to the non-locality properties which are so important in fractional calculus, and it opens new gates for studying non-locality phenomena in fractional variational principles and control theory. We frequently used this new formula in proving the later results of this paper.

In section \ref{sec-ode}, we constructed solutions for certain classes of fractional ordinary differential equations, both of Riemann--Liouville type and of Caputo type. Although the differential equations considered are all relatively simple, these results provide a window into the broader field of differential equations defined using fractional derivatives with Mittag-Leffler kernel, and they may be extended later on. Differential equations in the AB model have already found applications in e.g. diffusion processes \cite{atangana} and electrical circuits \cite{gomez-aguilar2}, so methods to solve them analytically will surely come in useful.

In section \ref{sec-semigroup}, we found conditions for the semigroup property -- one of the most important things to consider in \textit{any} definition of fractional derivatives and integrals -- to hold with the new definition involving Mittag-Leffler kernels, and established the perhaps surprising result that it almost never holds under this definition.

In sections \ref{sec-product} and \ref{sec-chain}, we used the new series formulae established in section 1 to prove extensions of the product rule and chain rule to the new scenario. This worked more or less exactly as expected -- albeit resulting in more complicated product and chain rule formulae than the original, but this is only to be expected in fractional calculus. These results, like the new series formula of section \ref{sec-series}, will be useful for numerical computation of the AB derivatives of many specific functions. We also demonstrated how the chain rule has direct applications in fractional dynamical systems, and how we can use it to construct a new fractional mechanics.

All of these results, rigorously proved for all functions which are differentiable according to the new definition, form part of the foundations for the theory of AB differintegrals. They may be used in proving many later results in this field, and for solving more advanced differential equations defined using the AB formula.

One possible direction in which to generalise the results of this paper is by looking at higher-order derivatives, the case where $\alpha>1$. The original publication \cite{atangana} which introduced the AB definition of differintegration only defined it for $0<\alpha<1$, but in the year since then, generalisations to $\alpha>1$ have been considered, see e.g. \cite{abdeljawad3}. It would be interesting to see how the theorems proved above can be extended to these cases, and this should be possible since the higher-order derivatives are defined simply by combining standard differentiation and integration with Definitions \ref{ABRL-defn} and \ref{ABC-defn}.

Another direction in which we would like to extend is to consider more complicated differential equations than those solved thus far, with the hope of finding exact analytic solutions if possible.

\section*{Acknowledgements}
The second author's research was supported by a grant from the Engineering and Physical Sciences Research Council, UK.

\end{document}